\documentclass[a4paper]{amsart}%[reqno]{amsart}
\usepackage{color}
\usepackage{amssymb,amsmath,amsthm,amstext,amsfonts}
\usepackage[dvips]{graphicx}
\usepackage{psfrag}
\usepackage{url}
\usepackage{amsfonts}
\usepackage{amsmath}
\usepackage{mathrsfs}
\usepackage{graphicx}
\usepackage{amsmath,amsthm,amssymb,amscd}
\usepackage{color}
\usepackage{enumerate}

\usepackage{epsfig}

\makeatletter \@addtoreset{equation}{section} \makeatother

\renewcommand\thetable{\thesection.\@arabic\c@table}

\theoremstyle{plain}
\newtheorem{maintheorem}{Theorem}

\newtheorem{maincorollary}{Corollary}

\newtheorem{theorem}{Theorem}[section]

\newtheorem{lemma}{Lemma}[section]

\newtheorem{definition}{Definition}[section]

\newtheorem{Thm}{Theorem}[section]
\newtheorem{Lem}[Thm]{Lemma}
\newtheorem{Prop}[Thm]{Proposition}

\theoremstyle{remark}
\newtheorem{Def}[Thm] {Definition}
\newtheorem{Rem}[Thm] {Remark}

\long\def\begcom#1\endcom{}

\newcommand{\length}{\operatorname{\length}}

\newcommand{\Int}{\operatorname{int}}

\def\length{\operatorname{length}}

\newcommand{\bl} {\begin{lemma}}
\newcommand{\el} {\end{lemma}}
\newcommand{\bt} {\begin{theorem}}
\newcommand{\et} {\end{theorem}}

\newcommand{\bp}{\begin{proof}}
\newcommand{\ep}{\end{proof}}
\newcommand  {\ee} {\end{equation}}
\newcommand  {\beq} {\begin{eqnarray*}}
\newcommand  {\eeq} {\end{eqnarray*}}

\newcommand  {\bd} {\begin{definition}}
\newcommand  {\ed} {\end{definition}}

\def\ep{\noindent{\hfill $\Box$}}

\begin{document}

%\title{Saturated sets of dynamical systems with non-uniform specification}
%\title{Different levels of recurrent scrambled set \\ for dynamic system with specification }
\title{Distributional chaos in multifractal analysis, recurrence and transitivity} % for dynamic systems with specification }

\author{An Chen and Xueting Tian}

\address{Xueting Tian, School of Mathematical Sciences,  Fudan University\\Shanghai 200433, People's Republic of China}
\email{xuetingtian@fudan.edu.cn}
\urladdr{http://homepage.fudan.edu.cn/xuetingtian}

\address{An Chen, School of Mathematical Sciences,  Fudan University\\Shanghai 200433, People's Republic of China}
\email{15210180001@fudan.edu.cn}
%\renewcommand{\baselinestretch}{1.2}
%\large\normalsize
%\date{\today}

\begin{abstract}

%In this paper, we discuss some levels of recurrence and  uncountable DC1-scrambled set.
There are lots of results to study dynamical complexity on irregular sets and level sets of ergodic average from the perspective of density in base space,  Hausdorff dimension, Lebesgue positive measure, positive or full topological entropy (and topological pressure) etc.. However, it is unknown from the viewpoint of chaos. There are lots of results on the relationship of positive topological entropy and various chaos but it is known that positive topological entropy does not imply a strong version of chaos called DC1 so that it is non-trivial to study DC1 on irregular sets and level sets.  In this paper we will show that 
for dynamical system with specification property,  there exist  uncountable DC1-scrambled subsets in irregular sets and level sets. On the other hand, we  also prove that several recurrent levels of points with different recurrent frequency   all have    uncountable DC1-scrambled subsets.   The main technique established to prove above  results is that there exists  uncountable DC1-scrambled subset in saturated sets. 

%different levels of recurrent frequency points all have  DC1 chaotic behavior. Further, we show that these phenomena still happen when restricted on multifractal analysis such as the irregular set or level set.
\end{abstract}

%which is the subset of recurrent points set but disjoint with the set of quasi-weakly almost periodic points. Furthermore, $f$ also possesses an   uncountable DC1-scrambled set which is the subset of the quasi-weakly almost periodic points but disjoint with the set of weakly almost periodic points.
%\footnote {%$^{*}$ 

%%\subjclass[2010]{37D35, 37D20, 37D25, 37C50}
\keywords{Irregular set and level set, Recurrence and Transitivity, Specification, Distributional chaos, Scrambled set}
%\tableofcontents
\subjclass[2010] {  37C50;  37B20;  37B05; 37D45; 37C45.   }
\maketitle
%\tableofcontents
%%%%%%%%%%%%%%%%%%%%%%%%%%%%%%%%
\section{Introduction}

Throughout this paper, let $(X,d)$ be a nondegenerate$($i.e, 
with at least two points$)$ compact metric space, and $f:X \rightarrow X$ is a continuous map. $(X,f)$ is called a dynamical system.

\subsection{Multifractal Analysis}
The theory of multifractal analysis is a subfield of the dimension theory of dynamical systems. Briefly, multifractal analysis studies the dynamical complexity of
the level sets of the invariant local quantities obtained from a dynamical system.  
There are lots of results to study dynamical complexity on irregular sets and level sets of ergodic average from the perspective of density in base space,  Hausdorff dimension, Lebesgue positive measure, positive or full topological entropy (and topological pressure) etc., for example, see \cite{Pesin-Pitskel1984,Barreira-Schmeling2000,  Pesin1997,   CKS,TDbeta,DOT,Bar2011,TV vp,FFW} (for topological entropy or Hausdorff dimension), \cite{Thompson2009,Thompson2008} (for topological pressure), \cite{Takens,KS} (for Lebesgue positive measure) and references therein.  However, it is unknown from the viewpoint of chaos. From chaos theory, we know that Li-Yorke chaotic and distributional chaotic are also good ways to describe the dynamical complexity. 
%For example, one can consider Birkhoff averages, Lyapunov exponents, pointwise dimensions or Bowen entropies\cite{BG,Ba-Sau-TAMS,Clim,Pesin1997}. The following definitions derive from the consideration of Birkhoff averages.
  In this paper, we firstly study  dynamical complexity of  irregular set and level sets  in the viewpoint of a strong chaotic property called DC1. Pikula showed in  \cite{Pik} that positive topological entropy does not imply DC1 so that it is not expected to show DC1 of irregular sets and level sets by using the results in \cite{Pesin-Pitskel1984,Barreira-Schmeling2000,   Ba-Sau-TAMS, TV vp} that irregular set and level sets carry  positive (and full) topological entropy.

The notion of chaos was first introduced in mathematic language by Li and Yorke in \cite{LY} in 1975. For a dynamical system $(X,f)$, they defined that $(X,f)$ is Li-Yorke chaotic if there is an uncountable scrambled set $S\subseteq X$, where $S$ is called a scrambled set if for any pair of distinct two points $x,y$ of $S$, $$\liminf_{n\to +\infty}d(f^nx,f^ny)=0,\ \limsup_{n\to +\infty}d(f^nx,f^ny) 
>0.$$  Since then, several refinements of chaos have been introduced and extensively studied. One of the most important extensions of the concept of chaos in sense of Li and Yorke is distributional chaos as introduced in \cite{SS1994}. The stronger form of chaos has three variants: DC1(distributional chaotic of type 1), DC2 and DC3 (ordered from strongest to weakest).   In this paper, we focus on DC1. Readers can refer to \cite{Dwic,SS,SS2} for the definition of DC2 and DC3 and see \cite{AK,OS,BrH,Dev,BGKM,Kan,Oprocha2009,BHS} and references therein for  related topics on chaos theory if necessary. 
A pair $x,y\in X$ is DC1-scrambled if the following two conditions hold:
$$\forall t>0,\ \limsup_{n\to \infty}\frac{1}{n}|\{i\in [0,n-1]:\ d(f^i(x),f^i(y))<t\}|=1,$$
$$\exists t_0>0,\ \liminf_{n\to \infty}\frac{1}{n}|\{i\in [0,n-1]:\ d(f^i(x),f^i(y))<t_0\}|=0.$$ In other words, the orbits of $x$ and $y$ are arbitrarily close with upper density one, but for some distance, with lower density zero. 

\begin{Def}
 A set $S$ is called a    DC1-scrambled set if   any pair of  distinct points in $S$ is DC1-scrambled.
\end{Def}

\subsubsection{DC1 in Irregular set }

For a continuous function $\varphi$ on $X$, define the \emph{$\varphi-$irregular set} as
\begin{eqnarray*}
  I_{\varphi}(f) &:=& \left\{x\in X: \lim_{n\to\infty}\frac1n\sum_{i=0}^{n-1}\varphi(f^ix) \,\, \text{ diverges }\right\}.
 % R_{\varphi}(a)^* &=& \{x\in X~|~\lim_{n\to\infty}\sum_{i=0}^{n-1}\varphi(f^ix)=a,  ~x\notin\omega(x)\};\\
 % R_{\varphi}(a)^{\#} &=&  \{x\in X~|~\lim_{n\to\infty}\sum_{i=0}^{n-1}\varphi(f^ix)=a,  ~x\notin\omega(x)~\text{and}~\omega(x)~\text{is minimal}\}.
\end{eqnarray*} 
$\varphi$-regular set and the irregular set, the union of  $I_{\varphi}(f)$ over all continuous functions of $\varphi$ (denoted by $IR(f)$), arise    in the context of multifractal analysis and have been studied a lot, for example,  see \cite{Pesin-Pitskel1984,Barreira-Schmeling2000,  Pesin1997,   CKS,Thompson2008,DOT}.  The irregular points  are also called points with historic behavior, %if   there is a continuous function such that $ \lim_n\frac{1}{n}\sum_{i=0}^{n-1}\varphi(f^{i}(x))$  does not exist,
  see \cite{Ruelle,Takens}. %The set of points with historic behavior is also called  irregular set, denoted by $IR(f),$ 
 From Birkhoff's ergodic theorem, the irregular set is not detectable from the point of view of any invariant measure.
 %From the viewpoint of ergodic theory,   the irregular points are negligible by Birkhoff Ergodic Theorem.  
 However, the irregular set may have strong dynamical complexity in sense of Hausdorff dimension, Lebesgue positive measure,  topological entropy and topological pressure etc..   
  %However,  they describe the points with same asymptotic behavior in the sense of ergodic average divergence.   
   Pesin and Pitskel \cite{Pesin-Pitskel1984} are the first to notice the phenomenon of the irregular set
carrying full topological entropy in the case of the full shift on two symbols.   There are lots of advanced results to show that the irregular points can carry full entropy in symbolic systems, hyperbolic systems, non-uniformly expanding or hyperbolic systems,  and systems with specification-like or shadowing-like properties, for example, see \cite{Barreira-Schmeling2000,   Pesin1997,  CKS,Thompson2008,DOT,LLST,TianVarandas}. For topological pressure case see \cite{Thompson2008} and for Lebesgue positive measure see \cite{Takens,KS}.  Now let us state our first main theorem to study dynamical complexity of irregular set from the perspective of DC1. 
%However, it is an increasingly well-known phenomenon that the irregular set can be large from the point of view of dimension theory\cite{BaL}. Takens and Verbitskiy have obtained multifractal analysis results for the class of maps with specification, using topological entropy as the dimension characteristic\cite{TV}. Later, many authors considered the irregular set and their classifications of maps with specification and proved that they all carry the full entropy, for example, see \cite{Thompson2008,CKS,T16}. 

%A map $f$ is called distributional chaotic of type 1(DC1 chaotic for brevity), if there is an   uncountable DC1-scrambled set $S\subseteq X$.

%Recently Thompson in [12, 13] discuss the topological entropy of irregular set for map with
%(almost) specification and show that the irregular set has full topological entropy
%(or pressure), which is inspired from [9] by Pfister and Sullivan and [11] by Takens
%and Verbitskiy. Our aim of thi

%We will show that DC1-chaotic appear in $I_{\varphi}(f)$.
%  and their intersection with recurrent levels.

\begin{maintheorem}\label{maintheorem-irregular}
Suppose that $(X,f)$ has specification property, $\varphi$ is a continuous function on $X$ and $I_{\varphi}(f)\neq\emptyset$. Then  there is an uncountable DC1-scrambled subset in $I_{\varphi}(f)$.  
\end{maintheorem}

% of irregular set.

%So question comes that what kind of dynamical system is Li-Yorke chaotic, or distributional chaotic and how big it is in the sense of measure or topological entropy? In 2002, Blanchard, Glasner, Kolyada and Alejandro proved that positive topological entropy implies Li-Yorke chaos in \cite{BGKM}. Afterwards, Pikula showed in  \cite{Pik} that positive topological entropy does not imply DC1 and Downarowicz proved in \cite{Dwic} that  the  Sm$\acute{\text{i}}$tal conjecture which states that positive topological entropy implies DC2. In 2008, Oprocha and $\breve{\text{S}}$tef$\acute{\text{a}}$nkov$\acute{\text{a}}$ shows in \cite{OS} that the   uncountable DC1-scrambled set is dense in $X$ when $(X,f)$ has a weaker form of specification. 

%In present paper we study the DC1 chaotic in different levels of recurrent frequency points, and we proved that these recurrent points sets of different levels all have the dense   uncountable DC1-scrambled set. 

\subsubsection{DC1 in Level sets}
Level sets is a natural concept to slice points with convergent Birkhoff¡¯s average operated by some continuous function, regarded as the multifractal decomposition \cite{Clim,FH}. Let $\varphi:X\rightarrow \mathbb{R}$  be a continuous function. For any $a\in  L_\varphi,  $
%denote
%$$t_a=\sup_{\mu\in \mathcal M_{f}(X)}\left\{h_\mu:\,  \int\varphi d\mu=a\right\}$$
%and
consider the level set
\begin{eqnarray*}
  R_{\varphi}(a) &:=& \left\{x\in X: \lim_{n\to\infty}\frac1n\sum_{i=0}^{n-1}\varphi(f^ix)=a\right\}.
 % R_{\varphi}(a)^* &=& \{x\in X~|~\lim_{n\to\infty}\sum_{i=0}^{n-1}\varphi(f^ix)=a,  ~x\notin\omega(x)\};\\
 % R_{\varphi}(a)^{\#} &=&  \{x\in X~|~\lim_{n\to\infty}\sum_{i=0}^{n-1}\varphi(f^ix)=a,  ~x\notin\omega(x)~\text{and}~\omega(x)~\text{is minimal}\}.
\end{eqnarray*}
Denote $R_\varphi=\bigcup_{a\in L_\varphi}R_{\varphi}(a)$, then $R_\varphi$ represents the regular points for $\varphi$. Many authors have considered the entropy of the $R_{\varphi}(a)$. For example, Barreira and Saussol proved in \cite{Ba-Sau-TAMS} that the following properties for a dynamical system $(X,f)$ whose function of metric entropy is upper semi-continuous.  Consider a H$\ddot{\text{o}}$lder continuous function $\phi$ (see~\cite{Bar2011,BD} for almost additive functions with tempered variation) which has a unique equilibrium measure, then for any constant $a\in\Int (L_\phi)$ % where $\Int(L_\phi)$ denotes the interior of $L_\phi$, it satisfies
 \begin{equation}\label{eq-1} h_{top}(R_\varphi(a))=t_a,
 \end{equation}
where $$t_a=\sup_{\mu\in \mathcal M_f(X)}\left\{h_\mu:\,  \int\varphi d\mu=a\right\},$$
$h_{top}(R_\varphi(a))$ denotes the entropy of $R_\varphi(a)$, $h_\mu$ denotes the measure entropy of $\mu$.
For $\phi$ being an arbitrary continuous function (hence there may exist more than one equilibrium measures), \eqref{eq-1} was established by Takens and Verbitski \cite{TV vp} under the assumption that $f$ has the specification property. This result was further generalized by Pfister and Sullivan  \cite{PS2} to dynamical systems with $g$-product property(see \cite{Thompson2009,T16} for more related discussions).  The method used in \cite{BD,Ba-Sau-TAMS} mainly depends on  thermodynamic formalism such as differentiability of pressure function while the method in \cite{TV vp,PS2} is a direct approach by constructing fractal sets. Here we consider the distributional chaotic of $R_\varphi(a)$ and $R_\varphi$.   Let  $\mathcal M(X)$, $\mathcal M_f(X)$, $\mathcal{M}^{e}_{f}(X)$ denote the space of probability measures, $f$-invariant, $f$-ergodic probability measures respectively. For a continuous function $\varphi$ on $X$, denote
$$L_\varphi=\left[\inf_{\mu\in \mathcal M_{f}(X)}\int\varphi d\mu,  \,  \sup_{\mu\in \mathcal M_{f}(X)}\int\varphi d\mu\right]~\textrm{and}~Int(L_\varphi)=\left(\inf_{\mu\in \mathcal M_{f}(X)}\int\varphi d\mu,  \,  \sup_{\mu\in \mathcal M_{f}(X)}\int\varphi d\mu\right).  $$
  %For a continuous function $\varphi$ on $X$,   d
  Note that if $I_{\varphi}(f)\neq\emptyset$,  then  $Int(L_\varphi)\neq \emptyset$. 
The inverse is also true if the system has specification property, see \cite{Thompson2008} (see \cite{TDbeta} for the case of almost specification), and it is easy to check the continuous functions with $Int(L_\varphi)\neq \emptyset$  form an open and dense subset in the space of continuous functions so that so do the functions with $I_{\varphi}(f)\neq\emptyset$ if the system has specification property or almost specification. 
%   and their intersection with recurrent levels.

\begin{maintheorem}\label{maintheorem-levelsets}
Suppose that $(X,f)$ has specification property, $\varphi$ is a continuous function on $X$ and $Int(L_\varphi)\neq\emptyset$. Then for any $a\in Int(L_\varphi)$,   there is an uncountable DC1-scrambled subset in $R_{\varphi}(a)$.
\end{maintheorem}

As a corollary, there are uncountable number of disjoint uncountable DC1-scrambled subsets. 

\begin{maincorollary}
 Suppose that $(X,f)$ has specification property but is not uniquely ergodic. Then there exist a collection of subsets of $X$, $\{S_\alpha\}_{\alpha\in (0,1)}$ such that \\
 (1). For any $0<\alpha_1<\alpha_2<1,$ $S_{\alpha_1}\cap S_{\alpha_2}=\emptyset$, and \\
 (2). For any $\alpha\in (0,1)$, $S_\alpha$ is an uncountable DC1-scrambled set.
\end{maincorollary}

Let us explain why this result holds. By assumption there are two different invariant measures $\mu,\nu$ so that by weak$^*$ topology there exists a continuous function $\phi$ such that $\int\phi d\mu\neq \int \phi d\nu.$ Thus $Int(L_\phi)\neq\emptyset$. Let $\varphi:= \frac1L(\phi-\inf_{\mu\in \mathcal M_{f}(X)}\int\phi d\mu)$ where $L$ denotes the length of interval $L_\phi$. Then $Int(L_\varphi)=(0,1)$ and Theorem \ref{maintheorem-levelsets} implies this corollary since $R_\varphi(a)\cap R_\varphi(b)=\emptyset$ if $a\neq b.$ %From entropy-dense of \cite{PS} (or see \cite{T16}), $(X,f)$ has specification property

\begin{Thm}\label{Th2018May-D-regularphi}
Suppose that $(X,f)$ has specification property, $\varphi$ is a continuous function on $X$. Then    there is an uncountable DC1-scrambled subset in $R_{\varphi}$.
\end{Thm}

Let us explain why Theorem \ref{Th2018May-D-regularphi} holds. If $Int(L_\varphi)\neq\emptyset$, then one can get this from Theorem \ref{maintheorem-levelsets} by taking one $a\in Int(L_\varphi)$ since $R_\varphi(a)\subseteq R_\varphi$. On the other hand, $Int(L_\varphi)=\emptyset$, then $R_\varphi=X$ and one can get this result by  \cite{OS} (or see \cite{Oprocha2007}). 
% from Theorem \ref{maintheorem-recurrent}. %We will prove Theorems \ref{maintheorem-irregular} and \ref{maintheorem-levelsets} in Section \ref{section-mainproof}.

\subsection{DC1 in  recurrence}

%The distributional chaos was defined imitating the Li-Yorke chaos. In consideration of this paper's purpose, here we just give the definition of distributional chaos of type 1(i.e the DC1). Readers can refer to \cite{SS2}\cite{LR} for more information about DC2 and DC3.
In classical study of dynamical systems, an important concept is recurrence. Recurrent points such as periodic points, minimal points are typical objects to be studied. It is known that whole recurrent points set has full measure for any invariant measure under $f$ and minimal points set is not empty\cite{Fur}.
%Poincar$\acute{e}$ proved that whole recurrent points set has full measure for any invariant measure under $f$. Birkhoff
A fundamental question in dynamical systems is to search the existence of periodic points. For systems with Bowen' specification(such as topological mixing subshifts of finite type and topological mixing uniformly hyperbolic systems), the set of periodic points is dense in the whole space \cite{Sig}. Further, many people pay attention to more refinements of recurrent points according to the 'recurrent frequency' such as weakly almost periodic points and quasi-weakly almost periodic points and measure them\cite{HYZ,ZH}. In \cite{HTW,T16} the authors considered various recurrence and showed many different recurrent levels carry strong dynamical complexity from the perspective of topological entropy. In present paper, one of our aim is to consider these different recurrent levels from the perspective of chaos.

For any $x \in X$, the orbit of $x$ is $\{f^nx\}_{n=0}^\infty$,   denoted by $orb(x,f)$. The $\omega$-limit set of $x$ is the set of all limit points of $orb(x,f)$,   denoted by $\omega_f(x)$. %If $\omega_f(x)=X$, we say $x$ is a transitive point of $f$. We denote $Trans$ the collection of all the transitive points of $f$.

\begin{Def} A point $x \in X$ is \textbf{recurrent}, if $x \in \omega_f(x)$. If $\omega_f(x)=X$, we say $x$ is a transitive point of $f$.   A point $x \in X$ is \textbf{almost periodic}, if for any open neighborhood $U$ of $x$, there exists $N \in \mathbb{N}$ such that $f^k(x) \in U$ for some $k \in [n, n+N]$ for every $n \in \mathbb{N}$. A point $x$ is called periodic, if there exists natural number $n$ such that $f^n(x)=x.$
\end{Def} 
 We denote the sets of all recurrent points, transitive points,  almost periodic points  and periodic points by $Rec$, $Trans$, $  AP$ and $Per$ respectively.  %\subsection{Chaotic in recurrence}
%\begin{Thm}\label{maintheorem1}
%Suppose that $(X,f)$ has specification property, then for any open set $U\subseteq X$, $f$ possesses {  two?????????????}    uncountable DC1-scrambled set $S_1,S_2\subseteq U$ such that $S_1\subseteq  Rec\backslash QW(f)$ $S_2\subseteq   QW(f)\backslash W(f).$
%\end{Thm}
 Now we recall some notions of   recurrence by using density. Let $S \subseteq \mathbb{N}$, we denote
$$\overline{d}(S):=\limsup_{n\to\infty}\frac{|S\cap\{0,1,\cdots,n-1\}|}{n},\ \ \underline{d}(S):=\liminf_{n\to\infty}\frac{|S\cap\{0,1,\cdots,n-1\}|}{n},$$ 
$$B^*(S):=\limsup_{|I|\to\infty}\frac{|S\cap I|}{|I|},\ \ B_*(S):=\liminf_{|I|\to\infty}\frac{|S\cap I|}{|I|},$$
where $|A|$ denotes the cardinality of the set $A$. They are called   the upper density of $S$  and   the lower density, % of $S$ respectively.
% If $\overline{d}(S) = \underline{d}(S) = d$, We call that $S$ has the density of $d$,which we denotes by $d(S)$. 
%We denote
%They are called  
 Banach upper density  and  Banach lower density of $S$ respectively.
Let $U,V \subseteq X$ be two nonempty open sets and $x \in X$. Define sets of visiting time
$$N(U,V):=\{ n \ge 1|U \cap f^{-n}(V) \not= \emptyset\}\ and\ N(x,U):=\{n \ge 1|f^n(x) \in U\}.$$
%The system $(X,f)$ is transitive if $N(U,V)\neq\emptyset$ holds for any $U,V \subseteq X$.
\begin{Def}
A point $x\in X$ is called Banach upper recurrent, if $N(x,B_\varepsilon(x))$ has positive Banach upper density where $B_\varepsilon(x)$ denotes the ball centered at $x$ with radius $\varepsilon$. Similarly, one can define the Banach lower recurrent, upper recurrent, and lower recurrent.(see \cite{HTW})
\end{Def}
Let $BR$ denote the set of all Banach upper recurrent points and let $QW,W$ denote the set of upper recurrent points and lower recurrent points respectively. Note that  $  AP$ coincides with the set of all Banach lower recurrent points. From \cite{HYZ,YYW,ZH,PW} $W,QW,BR,Rec$ all have full measure for any invariant measure but $AP$ maybe not. Note that $$  AP \subseteq  W \subseteq  QW \subseteq  BR \subseteq Rec.$$  So the recurrent set can be decomposed into several disjoint `periodic-like' recurrent level sets which reflect different recurrent frequency: 
$$ Rec= AP\sqcup (W\setminus AP)\sqcup (QW\setminus W) \sqcup (BR\setminus QW)\sqcup (Rec\setminus BR).$$  A question appeared in \cite{T16} is that 
  $$\text{\it How much diﬀerence are there between these `periodic-like' recurrences?}  $$ 
 One main basic idea firstly considered in \cite{T16} is to search which  recurrent level set carries the same dynamical complexity as the whole system, for example, by using topological entropy since $W,QW,BR,Rec$ all carry full topological entropy as the whole space. It was showed that these recurrent level sets except $Rec\setminus BR$ all have full topological entropy   studied in \cite{T16} for $QW\setminus W$ and $W\setminus AP$, \cite{HTW} for $BR\setminus QW$, \cite{DongTian2017} for $AP$. From \cite{Oprocha2007} Oprocha proved that there exists an uncountable DC1-scrambled subset in $Rec\setminus AP.$ Recall that Pikula showed in  \cite{Pik} that positive topological entropy does not imply DC1. 
 Thus, motivated by these results we can also ask the similar question from the perspective of chaos.  
%The entropy estimate on  $QW\setminus AP =(QW\setminus W)\cup (W\setminus AP)$, $BR\setminus QW$, $   AP$ and their classifications have been discussed in \cite{T16}, \cite{HTW} and \cite{DongTian2017} respectively. In other words, these recurrent levels all carry strong dynamical complexity from  the perspective of full or positive  topological  entropy. In present paper we discuss these recurrent levels from chaotic perspective which is another characteristic to describe the dynamical complexity, t
That is,  whether there is an   uncountable DC1-scrambled set in every recurrent level set of $Rec\setminus BR$, $BR\setminus QW$, $QW\setminus W$, $W\setminus AP$ and $AP$. We will mainly show there are  uncountable $DC1$-scrambled subsets in $BR\setminus QW$ and $QW\setminus W$ if the system has specification property (and we also discuss  uncountable $DC1$-scrambled subset in  $W\setminus AP$ under more assumptions and uncountable 
 $DC2$-scrambled subset in $AP$ in the last section).

%\begin{Def}
%For a collection of subsets $Z_1,Z_2,\cdots,Z_k\subseteq X(k\geq 2)$, we denote GS$\{Z_1,\\
%Z_2,\cdots,Z_k\}=\{Z_{2}\setminus Z_1,Z_{3}\setminus Z_2,\cdots,Z_k\setminus Z_{k-1}\}$ the gap sets of the sequence. We say $\{Z_i\}_{i=1}^k$ has  uncountable DC1-scrambled gap with respect to $Y(\subseteq X)$ if $S\cap Y$ possesses an   uncountable DC1-scrambled set for any $S\in \mathrm{GS}\{Z_1,Z_2,\cdots,Z_k\}$.
%\end{Def}

\begin{maintheorem}\label{maintheorem-recurrent}
Suppose that $(X,f)$ has specification property but is not uniquely ergodic. Then there exist  uncountable DC1-scrambled subsets in $  QW\setminus W$ and $  BR\setminus QW.$ Moreover, the points in these subsets can be chosen transitive. 
\end{maintheorem}
%It is still unknown whether this result is also hold for $Rec\setminus BR$.

We will prove Theorem  \ref{maintheorem-recurrent} in Section \ref{section-mainproof-combination}.

\begin{maincorollary}\label{maincorollary-transitive}
Suppose that $(X,f)$ has specification property. % but is not uniquely ergodic.
 Then there exist  uncountable DC1-scrambled subset  in $  Trans.$ 
 %Moreover, the points in these subsets can be chosen transitive. 
\end{maincorollary}

Let us explain why this result holds. By assumption if the system is not uniquely ergodic, then  it can be deduced from Theorem  \ref{maintheorem-recurrent}. Otherwise, the system is uniquely ergodic. By \cite{DYM} (or see \cite{HTW}) minimal points are dense in the whole space so that the system must also be minimal. In this case $Tran=AP=X$  so that one only needs to show uncountable DC1-scrambled in $X$ which is the result of  \cite{OS} (or see \cite{Oprocha2007}). 

\subsection{Combination of Multifractal Analysis and Recurrence}

We give a   DC1 result in combined sets of multifractal analysis and recurrence.  
\begin{maintheorem}\label{maintheorem-combination}
Suppose that $(X,f)$ has specification property, $\varphi$ is a continuous function on $X$ and $Int(L_\varphi)\neq\emptyset$. Then \\
(1) there exsit   uncountable DC1-scrambled subsets in   $ I_{\varphi} \cap (QW\setminus W) $ and $ I_{\varphi}\cap (BR\setminus QW) $ respectively.   \\
(2) for any $a\in Int(L_\varphi)$,   there exsit   uncountable DC1-scrambled subsets in   $ R_{\varphi}(a)\cap (QW\setminus W) $ and $ R_{\varphi}(a)\cap (BR\setminus QW) $ respectively.   \\
Moreover, the points in these subsets can be chosen transitive. 
\end{maintheorem}

 Theorem \ref{maintheorem-combination}  imply Theorems \ref{maintheorem-irregular} and  \ref {maintheorem-levelsets}   so that we only need to prove Theorem \ref{maintheorem-combination}    in Section \ref{section-mainproof-combination}. As a corollary of Theorem \ref{maintheorem-combination}, we state a following result.

\begin{maincorollary}\label{maincorollary-transitive-mixed}
Suppose that $(X,f)$ has specification property, $\varphi$ is a continuous function on $X$ and $Int(L_\varphi)\neq\emptyset$. Then  % but is not uniquely ergodic.
   there exist  uncountable DC1-scrambled subset  in $  Trans\cap I_{\varphi}.$ And for any $a\in Int(L_\varphi)$,   there exsit   uncountable DC1-scrambled subset  in   $ R_{\varphi}(a)\cap Trans$. % respectively. 
 %Moreover, the points in these subsets can be chosen transitive. 
\end{maincorollary}
 
%According to Theorem  \ref{Th-best-refinedversion},  \ref{Th-best-refinedversion},  \ref{Th-best-refinedversion} and \ref{Th-D-regularphi}, we immediately have some apparent corollaries, which means that strong chaos appear in transitive points, irregular sets, level sets and regular sets.
%\begin{maincorollary}
%Suppose that $(X,f)$ has specification property. Then \\ 
%(1) there is an   uncountable DC1-scrambled set $S\subseteq Trans.$
%\\
%(2) If $\varphi$ is a continuous function on $X$ and $I_{\varphi}(f)\neq\emptyset$. Then there is an   uncountable DC1-scrambled set $S\subseteq Trans\cap I_{\varphi}(f).$\\
%(3) If $\varphi$ is a continuous function on $X$ and $Int(L_\varphi)\neq\emptyset$. Then for any $a\in L_\varphi$, there is an   uncountable DC1-scrambled set $S\subseteq Trans\cap R_{\varphi}(a)$. 
%\\
%(4) If $\varphi$ is a continuous function on $X$. Then there is an   uncountable DC1-scrambled set $S\subseteq Trans\cap R_{\varphi}$.
%\end{maincorollary}

\subsection{DC1 in Recurrent Level Sets Characterized by Statistical $\omega-$limit Sets}\label{section-statistical}
Recently several concepts of statistical $\omega-$limit sets were introduced in \cite{DT} (some notions also see \cite{AAN,AF}). They also can describe different levels of recurrence and some cases coincide with above classifications of Banach recurrence. 
\begin{Def}
 For $x\in X$ and $\xi=\overline{d},   \,  \underline{d},   \,  B^*,  \,   B_*$,   a point $y\in X$ is called $x-\xi-$accessible,   if for any $ \varepsilon>0,  \,  N (x,  V_\varepsilon (y))\text{ has positive   density w.  r.  t.   }\xi,  $ where     $V_\varepsilon (x)$ denotes  the  ball centered at $x$ with radius $\varepsilon$.
 Let $$\omega_{\xi}(x):=\{y\in X\,  |\,   y\text{ is } x-\xi-\text{accessible}\}.  $$  For convenience,   it is called {\it $\xi-\omega$-limit set of $x$}.
  $\omega_{B_*}(x)$ is also called {\it syndetic center} of $x$.
 %namely,   $\omega_{\xi}(x)=\{y\in X:~\textrm{for any}~\eps>0,   \xi(N(x,  V_{\eps}(y)))>0\}$.
\end{Def}
With these definitions, one can immediately note that
\begin{equation}\label{omega-order}
  \omega_{B_*}(x)\subseteq \omega_{\underline{d}}(x)\subseteq \omega_{\overline{d}}(x)\subseteq \omega_{B^*}(x)\subseteq \omega_f(x).
\end{equation}
For any $x\in X$, if $\omega_{B_*}(x)=\emptyset,$ then from \cite{DT} we know that  $x$ satisfies one and only one of following twelve cases:
  \begin{description}
  \item[Case  (1)    ] \,   $  \emptyset=\omega_{B_*}(x)\subsetneq\omega_{\underline{d}}(x)= \omega_{\overline{d}}(x)= \omega_{B^*}(x)= \omega_f(x);$
  \item[Case  (1')    ]  \,  \,   $  \emptyset=\omega_{B_*}(x)\subsetneq\omega_{\underline{d}}(x)= \omega_{\overline{d}}(x)= \omega_{B^*}(x)\subsetneq\omega_f(x);$

  \item[Case  (2)    ]  \,  \,   $\emptyset=\omega_{B_*}(x)\subsetneq    \omega_{\underline{d}}(x)=
\omega_{\overline{d}}(x)  \subsetneq \omega_{B^*}(x)= \omega_f(x) ;$
\item[Case  (2')    ]  \,  \,   $\emptyset=\omega_{B_*}(x)\subsetneq  \omega_{\underline{d}}(x)=
\omega_{\overline{d}}(x)  \subsetneq \omega_{B^*}(x)\subsetneq \omega_f(x) ;$

  \item[Case  (3)    ]  \,  \,   $    \emptyset=\omega_{B_*}(x)=\omega_{\underline{d}}(x)\subsetneq
\omega_{\overline{d}}(x)= \omega_{B^*}(x)= \omega_f(x) ;$
  \item[Case  (3')    ] \,  \,   $\emptyset=\omega_{B_*}(x)=\omega_{\underline{d}}(x)\subsetneq
\omega_{\overline{d}}(x)= \omega_{B^*}(x)\subsetneq \omega_f(x) ;$
\item[Case  (4)    ]   \,  \,
$\emptyset=\omega_{B_*}(x)\subsetneq  \omega_{\underline{d}}(x)\subsetneq
\omega_{\overline{d}}(x)= \omega_{B^*}(x)= \omega_f(x) ;$
 \item[Case  (4')    ]  \,  \,
$\emptyset=\omega_{B_*}(x)\subsetneq  \omega_{\underline{d}}(x)\subsetneq
\omega_{\overline{d}}(x)= \omega_{B^*}(x)\subsetneq \omega_f(x) ;$

  \item[Case  (5)    ]  \,  \,   $\emptyset=\omega_{B_*}(x)=  \omega_{\underline{d}}(x) \subsetneq \omega_{\overline{d}}(x)\subsetneq \omega_{B^*}(x)= \omega_f(x);$
 \item[Case  (5')    ] \,  \,   $\emptyset=\omega_{B_*}(x)=  \omega_{\underline{d}}(x) \subsetneq \omega_{\overline{d}}(x)\subsetneq \omega_{B^*}(x)\subsetneq \omega_f(x);$
\item[Case  (6)    ]   \,  \,
$\emptyset=\omega_{B_*}(x)\subsetneq \omega_{\underline{d}}(x) \subsetneq \omega_{\overline{d}}(x)\subsetneq \omega_{B^*}(x)= \omega_f(x);$
 \item[Case  (6')    ]  \,  \,
$\emptyset=\omega_{B_*}(x)\subsetneq \omega_{\underline{d}}(x) \subsetneq \omega_{\overline{d}}(x)\subsetneq \omega_{B^*}(x)\subsetneq \omega_f(x).$

 %   \item[Case  (5).    ]  \,  \,   $ \emptyset\neq\omega_{B_*}(x)= \omega_{\underline{d}}(x)= \omega_{\overline{d}}(x)
%= \omega_{B^*}(x)= \omega_f(x) ;$
 %\item[Case  (5').    ]  \,  \,   $ \emptyset\neq\omega_{B_*}(x)= \omega_{\underline{d}}(x)= \omega_{\overline{d}}(x)
%= \omega_{B^*}(x)\subsetneq \omega_f(x) .  $
  \end{description}

\begin{Thm}\label{Th-E-statisticallanguage}
Suppose that $(X,f)$ has specification property but is not uniquely ergodic, then $\{x\in Rec|\ x\ satisfies \, \,
Case\ (i)\},\ i=2,3,4,5,6$ contains an  uncountable DC1-scrambled subset in $Trans.$ Further, if $\varphi$ is a continuous function on $X$ and $I_{\varphi}(f)\neq\emptyset$, then for any $a\in Int(L_\varphi)$, the recurrent level set of $\{x\in Rec|\ x\ satisfies\ Case\ (i)\}$  contains an  uncountable DC1-scrambled subset in $Trans\cap I_{\varphi}(f)$, $Trans\cap R_{\varphi}(a)$ and $Trans\cap R_{\varphi},$ respectively, $\ i=2,3,4,5,6  .$
\end{Thm}

We will prove this theorem in in Section \ref{section-mainproof-combination}. 
Case (1) is also known if the system has more assumptions, see the last section, but Cases (1')-(6') restricted on recurrent points all are still  unknown to have DC1   or weaker ones such as Li-Yorke chaos. 
Chaotic behavior in non-recurrent points and various non-recurrent levels by using above statistical $\omega$-limit sets will be discussed in another forthcoming paper.

\subsection{DC1  in Saturated sets}
 To show above results on irregular set, level sets and different recurrence, one main proof idea is motivated by Oprocha and $\breve{\text{S}}$tef$\acute{\text{a}}$nkov$\acute{\text{a}}$'s result  in \cite{OS} (or see \cite{Oprocha2009}) that there is an   uncountable DC1-scrambled subset   in $X$ when $(X,f)$ has   specification. One can construct corresponding   uncountable DC1-scrambled subset one by one but everyone needs a long construction proof so that it is not a good choice to do these constructions directly. Recall that in the case of entropy estimate on recurrent levels, one main technique chosen in  \cite{T16,HTW} is using (transitively) saturated property which can avoid to do a long construction proof for every considered object. So here we follow the way of \cite{T16,HTW} to give a DC1 result in saturated sets.  

  Given $x\in X$, denote $V_f(x)\subseteq \mathcal M_f(X)$ the set of all accumulation points of the empirical measures
$$
\mathcal{E}_n (x):=\frac1{n}\sum_{i=0}^{n-1}\delta_{f^i (x)},
$$
where $\delta_x$ is the Dirac measure concentrate on $x$.  
  % which is firstly established by \cite{PS2}. 
%Given $x\in X$, let $\omega_T(x)$ denote the $\omega-$limit set of $x$,    $M_x$ be the limit set of the empirical measures for $x$. Define
 % $ Rec =\{x\in X\,|x\in \omega_T(x)  \}$ and  $Tran   = \{x\in X|   \omega_T(x) = X\}.$
 %Given measure $m,$ let $S_m$ denote  the support of $m$.
 %For $A\subseteq X,$ let $h_{top}(A)$ denote the topological entropy of $A$ defined by Bowen in  \cite{Bowen1} ( see Section \ref{section-prepare})
  %and given an invariant measure $\mu,$ let  $h_\mu(T)$ denote its metric entropy of $\mu$ and let $S_\mu$ denote  the support of $\mu$..
 %\begin{Def}\label{Def-saturated}
The system $(X,f)$ is called to have {\it  saturated} property, if  for any  compact connected nonempty set $K \subseteq \mathcal M_f(X),$
\begin{eqnarray} \label{eq- saturated-definition}
 G_K\neq \emptyset \,\,\text{ and }\,\,h_{top} (T,G_K )=\inf\{h_\mu (T)\,|\,\mu\in K\},
%\text{ where } G^T_{K}=\{x\in Tran|\, M_x=K   \}.
\end{eqnarray}  where $G_{K}
=\{x\in X|\, V_f(x)=K
  \} $ (called saturated set), $h_{top}(A)$ denotes the topological entropy of $A$ defined by Bowen in  \cite{Bowen1973}  and  $h_\mu(T)$ denotes its metric entropy of $\mu$.   
%On the other hand, we say that the system $T$ has {\it transitively-saturated} property or $T$ is {\it transitively-saturated}, if above equality (\ref{eq- saturated-definition}) only holds for $U=X.$ In parallel, one can define   transitively-convex-saturated and transitively-single-saturated respectively.
   % \end{Def}
The existence of saturated sets is proved by Sigmund \cite{SigSpe} for systems with uniform hyperbolicity or specification and generalized to   non-uniformly hyperbolic systems in \cite{LST}. The property on entropy estimate was firstly established by Pfister and  
  Sullivan in \cite{PS2} and then was generalized to transitively-saturated version in \cite{HTW}, provided that the system has $g$-product property (which is weaker than specification) and uniform separation property (which is weaker than expansiveness). In this subsection we aim to establish DC1    in saturated sets. A point $x\in X$ is generic for some invariant measure $\mu$ means that $V_f(x)=\mu$(or equivalently, Birkhoff averages of all continuous map converge to the integral of $\mu$.)  Let $G_\mu$ denote the set of all generic points for $\mu$.

\begin{maintheorem}\label{maintheorem-DC1inSaturated}
Suppose that $(X,f)$ has specification and $K$ be a connected non-empty compact subset of $\mathcal M_f(X)$. If there is a $\mu\in K$ such that $\mu=\theta\mu_1+(1-\theta)\mu_2\ (\mu_1=\mu_2\mathrm{\ could\ happens})$ where $\theta\in[0,1]$, and $G_{\mu_1}$, $G_{\mu_2}$ both have distal pair. Then for any non-empty open set $U\subseteq X$, there exists an   uncountable DC1-scrambled set $S_K\subseteq G_K\cap U\cap Trans$.
\end{maintheorem}

We will prove this theorem in Section \ref{section-proof-saturated}. Since an ergodic measure with nondegenerate minimal support has two generic points as a distal pair, see Proposition \ref{distal-dense} below, one has a following result as a corollary of Theorem \ref{maintheorem-DC1inSaturated}.

\begin{maincorollary}Suppose that $(X,f)$ has specification. For any ergodic measure $\mu$,   if its support is nondegenerate and minimal, then  there exists an   uncountable DC1-scrambled set $S\subseteq Trans$ such that any point in $S$ is generic for $\mu.$

\end{maincorollary}
Here $\mu$ admits to have zero metric entropy. If the system is not minimal, then above set $S$ has zero measure for $\mu$, since $S\subseteq Trans$, $S_\mu\neq X$ and by Birkhoff ergodic theorem $\mu(S_\mu\cap G_\mu)=1.$

\section{Preliminaries}

%%%%%%%%%%%%%%%%%%%%%%%%%%%%%%%%
%\subsection{Distributional Chaos and Scrambled Set}
%The chaos was first introduced in mathematic by Li and Yorke in 1975. For $(X,f)$, they defined that if there is an uncountable scrambled set $S\subseteq X$, we call that $(X,f)$ is Li-Yorke chaotic, where $S$ is called a scrambled set if for any pair of distinct two points $x,y$ of $S$ $\liminf_{n\to +\infty}d(f^nx,f^ny)=0,\ \limsup_{n\to +\infty}d(f^nx,f^ny)>0$.
%
%The distributional chaos was defined imitating the Li-Yorke chaos. In consideration of this paper's purpose, here we just give the definition of distributional chaos of type 1(i.e the DC1). Readers can refer to \cite{SS2}\cite{LR} for more information about DC2 and DC3.

\subsection{Specification Properties}
Specification was first introduced by Bowen in  \cite{Bowen}. Before giving the definition, we make a notion that for $(X,f)$ and $x,y\in X,\ a,b\in\mathbb{N}$, we say $x$ $\varepsilon$-$traces$ $y$ on $[a,b]$ if $d(f^ix,f^{i-a}y)<\varepsilon\ \forall i\in[a,b]$. The following definition mainly refers to \cite{Sig,OS}.
\begin{Def}\label{definition of specification}
 We say $(X,f)$ has \textbf{strong specfication property}, if for
any $\varepsilon > 0$, there is a positive integer $K_\varepsilon$ such that for any integer s $\ge$ 2, any set $\{y_1,y_2,\cdots,y_s\}$ of $s$ points of $X$, and any sequence\\
$$0 = a_1 \le b_1 < a_2 \le b_2 < \cdots < a_s \le b_s$$ of 2$s$ integers with $$a_{m+1}-b_m \ge K_\varepsilon$$ for $m = 1,2,\cdots,s-1$, there is a point $x$ in $X$ such that the
following two conditions hold:\
\item[(a)] $x$ $\varepsilon$-$traces$ $y_m$ on $[a_m,b_m]$ for all positive integers m $\le$ s;
\item[(b)] $f^n(x) = x$, where $n = b_s + K_\varepsilon$.\\
If the periodicity condition (b) is omitted, we say that $f$ has \textbf{specification property}.
\end{Def}

\begin{Prop}\label{specification-entropy-dense}\cite{EKW}
Suppose that $(X,f)$ has specification property, then $\mathcal{M}_f^e(X)$ is dense in $  \mathcal M_f(X)$.
\end{Prop}

\begin{Prop}\label{specification-full-support-dense}\cite{Sig}
A dynamical system $(X,f)$ with specification property has measure with full support. Moreover, the set of such measure is dense in $ \mathcal M_f(X)$.
\end{Prop}

%\subsection{Minimality}
%A subset $A\subseteq X$ is called an invariant set of $f$, if $f(A)\subseteq A$. A subset $M\subseteq X$ is called a minimal set of $f$, if it is a nonempty closed invariant set of, and no proper subset of $M$ has this property. We called the point in a minimal set the minimal point. Recall that $M$ is a minimal set of $f$ iff $\omega_f(x)=M$ for all $x\in M$(see \cite{Gmin}). It is known that $x$ is a minimal point iff $x\in AP$. So, we denote $  AP$ also the  set of all minimal points.

%Given $(X,f)$,$(Y,g)$, if there is a continuous surjection $\pi:X\rightarrow Y$ with $\pi\circ f=g\circ \pi$, we say $f$ and $g$ are semiconjugate. The map $\pi$ is called a factor map, and $(Y,g)$ is called the factor system of $(X,f)$.
%\begin{Prop}\label{prop-factor-of-minimal}
%Suppose $(Y,g)$ is a factor system of $(X,f)$ with the factor map $\pi$ then
%\item[(1)] $fAP=  AP$;
%\item[(2)] $\textbf{  AP}(f^n)=  AP, \forall n\in\mathbb{N}$;
%\item[(3)] $\piAP=\textbf{  AP}(g)$.
%\end{Prop}

\subsection{Levels of Recurrence}

Let us recall some equivalent statements of recurrence  referring to \cite{HYZ,YYW,ZH,HTW}. For a measure $\mu$, define the support of $\mu$ by $S_\mu:=supp(\mu)=\{x\in X|\ \mu(U)>0\ $for any neighborhood $U$ of $x\}.$  
\begin{Prop}\label{prop1} \cite{HYZ}
For $(X,f)$, let $x \in Rec$. Then the following conditions are equivalent.
\item[(a)] $x \in  W$;
%\item[(b)] $\underline{d}(N(x,B(x,\varepsilon))) > 0\ for\ any\ \varepsilon > 0$;
\item[(b)] $x \in C_x = S_\mu\ for\ any\ \mu \in V_f(x)$;
\item[(c)] $S_\mu = \omega_f(x)\ for\ any\ \mu \in V_f(x).$
\end{Prop}
\begin{Prop}\label{prop2}\cite{HYZ}
For $(X,f)$, let $x \in Rec$. Then the following conditions are equivalent.
\item[(a)] $x \in  QW$;
%\item[(b)] $\overline{d}(N(x,B(x,\varepsilon))) > 0\ for\ any\ \varepsilon > 0;$
\item[(b)] $x \in C_x$;
\item[(c)] $C_x = \omega_f(x)$.
%\item[(e)] $\omega_f(x)=\overline{\cup_{\mu \in V_f(x)}S_\mu}$
\end{Prop}

%Readers can refer to \cite{HYZ} for the proof of the two propositions above.

\begin{Prop}\label{Trans-BR}
For $(X,f)$ with specification property, $x\in Trans$ implies $x\in BR$.
\end{Prop}
Proposition \ref{Trans-BR} is direct consequence by combining Proposition \ref{specification-full-support-dense} and  \cite[Lemma 4.3]{HTW}.

\section{ Proof of Theorem \ref{maintheorem-DC1inSaturated}} \label{section-proof-saturated}

  One main proof idea is motivated by Oprocha and $\breve{\text{S}}$tef$\acute{\text{a}}$nkov$\acute{\text{a}}$'s result  in \cite{OS} that there is   uncountable DC1-scrambled subset   in $X$ when $(X,f)$ has   specification. 
Before proof we introduce some basic facts and lemmas.

\subsection{Ergodic Average}
We write $\mathbb{N}=\{0,1,2,\cdots\}$ and $\mathbb{N}^+=\{1,2,\cdots\}$. If $r,s\in\mathbb{N},r\leq s$, we set $[r,s]:=\{j\in\mathbb{N}|\ r\leq j\leq s\}$, and the cardinality of a finite set $\Lambda$ is denoted by $|\Lambda|$. We set
$$
\langle f,\mu \rangle\ :=\ \int_Xfd\mu.
$$
There exists a countable and separating set of continuous functions $\{f_1,f_2,\cdots\}$ with $0\leq f_k(x)\leq 1$, and such that
$$
d(\mu,\nu)\ :=\ \parallel\mu-\nu\parallel\ :=\ \sum_{k\geq 1}2^{-k}\mid\langle f_k,\mu-\nu \rangle\mid
$$
defines a metric for the weak*-topology on $ \mathcal M_f(X)$. We refer to \cite{PS2} and use the metric on $X$ as following defined by Pfister and Sullivan.
$$
d(x,y) := d(\delta_x,\delta_y),
$$
which is equivalent to the original metric on $X$. Readers will find the benefits of using this metric in our proof later.
\begin{Lem}\label{measure distance}
For any $\varepsilon > 0,\delta >0$ and two sequences $\{x_i\}_{i=0}^{n-1},\{y_i\}_{i=0}^{n-1}$ of $X$ such that $d(x_i,y_i)<\varepsilon$ holds for any $i\in [0,n-1]$, then for any $J\subseteq \{0,1,\cdots,n-1\}$, $\frac{n-|J|}{n}<\delta$, one has:
\item[(a)] $d(\frac{1}{n}\sum_{i=0}^{n-1}\delta_{x_i},\frac{1}{n}\sum_{i=0}^{n-1}\delta_{y_i})<\varepsilon.$
\item[(b)] $d(\frac{1}{n}\sum_{i=0}^{n-1}\delta_{x_i},\frac{1}{|J|}\sum_{i\in J}\delta_{y_i})<\varepsilon+2\delta.$
\end{Lem}

Lemma \ref{measure distance} is easy to be verified and shows us that if any two orbit of $x$ and $y$ in finite steps are close in the most of time, then the two empirical measures induced by $x,y$ are also close.

\begin{Lem}\label{lemma-KK}
Suppose that $(X,f)$ has specification. Let $K$ be a connected non-empty compact subset of $\mathcal M_f(X)$ and $\mu\in K$. Then for any $\varepsilon>0$ there exists a $N_\varepsilon^\mu\in \mathbb{N}$ such that for any $\alpha\in K$, any $N> N_\varepsilon^\mu$ and any $M>N$, there is an $x\in X$ and $N^*>M$ such that
\begin{description}
\item[(a)] $\mathcal E_{n}(x)\in B(\mu,\varepsilon),\ \forall n\in[N_\varepsilon^\mu,N]$;

\item[(b)] $\mathcal E_{n}(x)\in B(K,\varepsilon),\ \forall n\in[N,N^*]$;

\item[(c)] $\mathcal E_{N^*}(x)\in B(\alpha,\varepsilon)$.
\end{description}
\end{Lem}
\begin{proof}
For any fixed $\varepsilon>0$, by Proposition \ref{specification-entropy-dense}, there exists $p^{\mu}\in X$ and $n^{\mu}\in \mathbb{N}$ such that $\mathcal E_n(p^{\mu})\in B(\mu, \varepsilon/6)$ holds for any $n\geq n^\mu$.  Set $N_\varepsilon^\mu:=n^\mu$, we will prove that such $N_\varepsilon^\mu$ makes this lemma true. Note that $K$ is connected, so for any $\alpha\in K$, we can find a sequence $\{\beta_1,\beta_2,\cdots,\beta_{m_\varepsilon}\}\subseteq K$ such that $d(\beta_{i+1},\beta_i)<\varepsilon,\ \forall i\in\{1,2,\cdots,m_\varepsilon-1\}$ and $\beta_1=\mu,\beta_{m_\varepsilon}=\alpha$. By Proposition \ref{specification-entropy-dense}, for any $i\in\{2,\cdots,m_\varepsilon\}$, there exists $p^{\beta_i}\in X$ and $n^{\beta_i}\in \mathbb{N}$ such that $\mathcal E_n(p^{\beta_i})\in B(\beta_i, \varepsilon/6)$ holds for any $n\geq n^{\beta_i}$. For any $N> N_\varepsilon^\mu$ and $M>N$, we choose $\{T_i\}_{i=1}^{2m_\varepsilon}$ with $T_i\in\mathbb{N}$ such that for $i\in\{1,\cdots,m_\varepsilon-1\}$
\begin{equation}\label{guji-1}
T_1=0,\ T_2=N.
\end{equation}
\begin{equation}\label{guji-2}
T_{2i+1}=T_{2i}+K_{\varepsilon/6}, \mathrm{where}\ K_{\varepsilon/6}\ \mathrm{defined\ in\ the}\ Definiton\ \ref{definition of specification}.
\end{equation}
\begin{equation}\label{guji-3}
\frac{\varepsilon}{12}(T_{2i}-T_{2i-1})>n^{\beta_{i+1}}.
\end{equation}
\begin{equation}\label{guji-4}
\frac{K_{\varepsilon/6}+T_{2i-1}}{T_{2i}-T_{2i-1}}<\frac{\varepsilon}{12}.
\end{equation}
So far, we have fixed $\{T_i\}_{i=1}^{2m_\varepsilon-1}$. We choose $T_{2m_\varepsilon}$ large enough such that
\begin{equation}\label{guji-5}
T_{2m_\varepsilon}\geq\mathrm{max}\{M,T_{2m_\varepsilon-1}+n^{\beta_{m_\varepsilon}}\}.
\end{equation}
\begin{equation}\label{guji-6}
\frac{T_{2m_\varepsilon-1}}{T_{2m_\varepsilon}}<\frac{\varepsilon}{12}.
\end{equation}
By (\ref{guji-2}), we can use specification property. So there is an $x\in X$ that $x$  $\varepsilon/6$-$traces$ $x^*$ on $[T_1,T_2]$ and $\varepsilon/6$-$traces$ $p^{\beta_i}$ on $[T_{2i-1},T_{2i}]$, $\forall i\in\{2,\cdots,m_\varepsilon\}$. Now we claim that such $x$ and $N^*=T_{2m_\varepsilon}$ satisfy the items $\mathbf{(a)(b)(c)}$. $\mathbf{(a)(c)}$ is easy to check by (\ref{guji-1})(\ref{guji-5})(\ref{guji-6}) and Lemma \ref{measure distance}. Here we check the $\mathbf{(b)}$. If $n\in (T_{2i},T_{2i+1})$ for some $i\in\{1,\cdots,m_\varepsilon-1\}$, we have
$$\frac{n-T_{2i}+T_{2i-1}}{T_{2i}-T_{2i-1}}<\frac{\varepsilon}{12},$$
by (\ref{guji-2})(\ref{guji-4}). So, by Lemma \ref{measure distance}, we have
\begin{align}\label{distance-guji-1}
d(\mathcal E_n(x),\beta_i) <& d(\mathcal E_n(x),\mathcal E_{T_{2i}-T_{2i-1}}(p^{\beta_i}))+d(\mathcal E_{T_{2i}-T_{2i-1}}(p^{\beta_i}),\beta_i)\notag\\
<& \frac{\varepsilon}{6}+2\cdot\frac{\varepsilon}{12}+\frac{\varepsilon}{6}\\
=& \frac{\varepsilon}{2}\notag.
\end{align}
If $n\in[T_{2i-1},T_{2i}]$ for some $i\in\{2,3,\cdots,m_\varepsilon\}$, we split this situation into the following two cases.

  {\bf Case 1:}  $\frac{n-T_{2i-1}}{T_{2i-2}-T_{2i-3}}<\frac{\varepsilon}{12}$. Then
  \begin{equation}\label{distance-guji-2}
  d(\mathcal E_n(x),\beta_{i-1})<\frac{\varepsilon}{6}+2\cdot(\frac{\varepsilon}{12}+\frac{\varepsilon}{12})+\frac{\varepsilon}{6}=\frac{2\varepsilon}{3},
  \end{equation}
  by Lemma \ref{measure distance} and (\ref{guji-4}).

  {\bf Case 2:}  $\frac{n-T_{2i-1}}{T_{2i-2}-T_{2i-3}}\geq\frac{\varepsilon}{12}$. If so, we have $n-T_{2i-1}>n^{\beta_i}$ by (\ref{guji-3}), which implies $\mathcal E_{n-T_{2i-1}}(p^{\beta_i})\in B(\beta_i,\varepsilon/6)$. We consider $d(\mathcal E_n(x),\beta_i)$ and $d(\mathcal E_n(x),\beta_{i-1})$.
\begin{align*}
d(\mathcal E_n(x),\beta_i) &= d(\frac{T_{2i-1}}{n}\mathcal E_{T_{2i-1}}(x)+\frac{n-T_{2i-1}}{n}\mathcal E_{n-T_{2i-1}}(f^{T_{2i-1}}x),\beta_i)\\
&\leq \frac{T_{2i-1}}{n}d(\mathcal E_{T_{2i-1}}(x),\beta_i)+\frac{n-T_{2i-1}}{n}d(\mathcal E_{n-T_{2i-1}}(f^{T_{2i-1}}x),\beta_i)\\
&\leq \frac{T_{2i-1}}{n}d(\mathcal E_{T_{2i-1}}(x),\beta_{i-1})+\frac{T_{2i-1}}{n}d(\beta_{i-1},\beta_{i})+\frac{n-T_{2i-1}}{n}d(\mathcal E_{n-T_{2i-1}}(f^{T_{2i-1}}x),\beta_i)\\
&< \frac{T_{2i-1}}{n}(\frac{\varepsilon}{6}+2\cdot\frac{\varepsilon}{12}+\frac{\varepsilon}{6})+\frac{T_{2i-1}}{n}\varepsilon+\frac{n-T_{2i-1}}{n}(\frac{\varepsilon}{6}+\frac{\varepsilon}{6})\\
&< \frac{\varepsilon}{2}+\frac{T_{2i-1}}{n}\varepsilon,
\end{align*}
 \begin{align*}
d(\mathcal E_n(x),\beta_{i-1}) &= d(\frac{T_{2i-1}}{n}\mathcal E_{T_{2i-1}}(x)+\frac{n-T_{2i-1}}{n}\mathcal E_{n-T_{2i-1}}(f^{T_{2i-1}}x),\beta_{i-1})\\
&\leq \frac{T_{2i-1}}{n}d(\mathcal E_{T_{2i-1}}(x),\beta_{i-1})+\frac{n-T_{2i-1}}{n}d(\mathcal E_{n-T_{2i-1}}(f^{T_{2i-1}}x),\beta_{i-1})\\
&\leq \frac{T_{2i-1}}{n}d(\mathcal E_{T_{2i-1}}(x),\beta_{i-1})+\frac{n-T_{2i-1}}{n}d(\mathcal E_{n-T_{2i-1}}(f^{T_{2i-1}}x),\beta_i)+\frac{n-T_{2i-1}}{n}d(\beta_{i},\beta_{i-1})\\
&< \frac{T_{2i-1}}{n}(\frac{\varepsilon}{6}+2\cdot\frac{\varepsilon}{12}+\frac{\varepsilon}{6})+\frac{n-T_{2i-1}}{n}(\frac{\varepsilon}{6}+\frac{\varepsilon}{6})+\frac{n-T_{2i-1}}{n}\varepsilon\\
&< \frac{\varepsilon}{2}+\frac{n-T_{2i-1}}{n}\varepsilon.
\end{align*}
So,
\begin{equation}\label{distance-guji-3}
\mathrm{min}\{d(\mathcal E_n(x),\beta_i),d(\mathcal E_n(x),\beta_{i-1})\}<\varepsilon.
\end{equation}
With the combination of (\ref{distance-guji-1}) (\ref{distance-guji-2}) (\ref{distance-guji-3}), one has $\mathbf{(b)}$. 
\end{proof}

\begin{Lem}\label{lemma-KKKK}
Suppose that $(X,f)$ has specification. Let $K$ be a connected non-empty compact subset of $\mathcal M_f(X)$ and $\mu\in K$. Then for any $\varepsilon>0$, there exists a $M_\varepsilon^\mu\in \mathbb{N}$ such that for any $\alpha\in K$ and any $M> M_\varepsilon^\mu$, there exist $t_2>t_1>M$ and $x\in X$ such that
\begin{description}
\item[(a)] $\mathcal E_{n}(x)\in B(\mu,\varepsilon),\ \forall n\in[M_\varepsilon^\mu,M]$;

\item[(b)] $\mathcal E_{n}(x)\in B(K,\varepsilon),\ \forall n\in[M,t_1]$;

\item[(c)] $\mathcal E_{t_1}(x)\in B(\alpha,\varepsilon)$;

\item[(d)] $\mathcal E_{n}(x)\in B(K,\varepsilon),\ \forall n\in[t_1,t_2]$;

\item[(e)] $\mathcal E_{t_2}(x)\in B(\mu,\varepsilon)$.
\end{description}
\end{Lem}

\begin{proof}
By Lemma \ref{lemma-KK}, for $\varepsilon/3$, we obtain $N_{\varepsilon/3}^\mu$ and $N_{\varepsilon/3}^{\alpha}$ such that for any $N_1>N_{\varepsilon/3}^\mu$, there is an $x_1$ and $N^*$ such that
\begin{equation}\label{KKKK-1}
N^*>\mathrm{max}\{N_1,\frac{K_{\varepsilon/3}+N_{\varepsilon/3}^{\alpha}}{\varepsilon/6}\},
\end{equation}
\begin{align*}
& \mathcal E_{n}(x_1)\in B(\mu,\varepsilon/3),\ \forall n\in[N_{\varepsilon/3}^\mu,N_1];\\
& \mathcal E_{n}(x_1)\in B(K,\varepsilon/3),\ \forall n\in[N_1,N^*];\\
& \mathcal E_{N^*}(x_1)\in B(\alpha,\varepsilon/3),
\end{align*}
and for
\begin{equation}\label{KKKK-2}
N_2>\mathrm{max}\{N_{\varepsilon/3}^{\alpha},\frac{N^*+K_{\varepsilon/3}}{\varepsilon/6}\},
\end{equation}
there exist $N^{**}>N_2$ and $x_2$ such that
\begin{align}\label{KKKK-3}
& \mathcal E_{n}(x_2)\in B(\alpha,\varepsilon/3),\ \forall n\in[N_{\varepsilon/3}^\alpha,N_2];\\
& \mathcal E_{n}(x_2)\in B(K,\varepsilon/3),\ \forall n\in[N_2,N^{**}];\notag\\
& \mathcal E_{N^{**}}(x_2)\in B(\mu,\varepsilon/3)\notag.
\end{align}
By specification property, we can obtain an $x\in X$ such that $x$ $\varepsilon/3$-$traces$ $x_1$ on $[0,N^*]$ and $\varepsilon/3$-$traces$ $x_2$ on $[N^*+K_{\varepsilon/3},N^*+K_{\varepsilon/3}+N^{**}]$. Now we consider $\mathcal E_{n}(x),\ n\in[N_{\varepsilon/3}^\mu,N^*+K_{\varepsilon/3}+N^{**}]$ and split into the following cases

  {\bf Case 1:}  When $n\in[N_{\varepsilon/3}^\mu,N^*]$, we have $d(\mathcal E_{n}(x),\mathcal E_{n}(x_1))<\varepsilon/3$. So
\begin{align*}
& \mathcal E_{n}(x)\in B(\mu,\varepsilon),\ \forall n\in[N_{\varepsilon/3}^\mu,N_1];\\
& \mathcal E_{n}(x)\in B(K,\varepsilon),\ \forall n\in[N_1,N^*];\\
& \mathcal E_{N^*}(x)\in B(\alpha,\varepsilon).
\end{align*}

  {\bf Case 2:}  When $n\in[N^*,N^*+K_{\varepsilon/3}+N_{\varepsilon/3}^\alpha]$, we have $d(\mathcal E_{n}(x),\mathcal E_{N^*}(x_1))<2\varepsilon/3$ by (\ref{KKKK-1}) and Lemma \ref{measure distance}. So $d(\mathcal E_{n}(x),\alpha)<\varepsilon$.

  {\bf Case 3:}  When $n\in[N^*+K_{\varepsilon/3}+N_{\varepsilon/3}^\alpha,N_2]$,
\begin{align*}
d(\mathcal E_{n}(x),\alpha) &= d(\frac{N^*+K_{\varepsilon/3}}{n}\mathcal E_{N^*+K_{\varepsilon/3}}(x)+\frac{n-N^*-K_{\varepsilon/3}}{n}\mathcal E_{n-N^*-K_{\varepsilon/3}}(f^{N^*+K_{\varepsilon/3}}x),\alpha)\\
&\leq \frac{N^*+K_{\varepsilon/3}}{n}d(\mathcal E_{N^*+K_{\varepsilon/3}}(x),\alpha)+\frac{n-N^*-K_{\varepsilon/3}}{n}d(\mathcal E_{n-N^*-K_{\varepsilon/3}}(f^{N^*+K_{\varepsilon/3}}x),\alpha).
\end{align*}
Note that $n-N^*-K_{\varepsilon/3}\geq N_{\varepsilon/3}^\alpha$ and $n\leq N_2$, then we have $d(\mathcal E_{n-N^*-K_{\varepsilon/3}}(f^{N^*+K_{\varepsilon/3}}x),\alpha)<\varepsilon$ by (\ref{KKKK-3}). So
$$d(\mathcal E_{n}(x),\alpha)<\frac{N^*+K_{\varepsilon/3}}{n}\varepsilon+\frac{n-N^*-K_{\varepsilon/3}}{n}\varepsilon=\varepsilon.$$

  {\bf Case 4:}  When $n\in[N_2,N^{**}]$. Note that $N^{**}>N_2>\frac{N^*+K_{\varepsilon/3}}{\varepsilon/6}$, so by Lemma \ref{measure distance}, we have
  $$d(\mathcal E_{n}(x),\mathcal E_{n-N^*-K_{\varepsilon/3}}(x_2))<2\varepsilon/3.$$
  Thus
  \begin{align*}
& \mathcal E_{n}(x)\in B(K,\varepsilon),\ \forall n\in[N_2,N^{**}];\\
& \mathcal E_{N^{**}}(x_2)\in B(\mu,\varepsilon).
\end{align*}
Set $M_\varepsilon^\mu=N_{\varepsilon/3}^\mu$, $M=N_1$ $t_1=N^*$ $t_2=N^{**}$, we finish the proof.
\end{proof}

For a dynamical system $(X,f)$, we say a pair $p,q\in X$ is distal if $\liminf_{i\to\infty}d(f^ip,f^iq)>0$. Obviously, $\inf\{d(f^ip,f^iq)|\ i\in\mathbb{N}\}>0$ if $p,q$ is distal. We say a subset $M\subseteq X$ has distal pair if there are distinct $p,q\in M$ such that $p,q$ is distal.

\begin{Lem}\label{lemma-mu}
Suppose that $(X,f)$ has specification. Suppose there are $\mu_1,\mu_2\in\mathcal M_f(X)$ such that $G_{\mu_1}$, $G_{\mu_2}$ have distal pair $(p_1,q_1)$, $(p_2,q_2)$ respectively. Let $\zeta=\mathrm{min}\{\inf\{d(f^ip_1,\\
f^iq_1)|\ i\in\mathbb{N}\},\inf\{d(f^ip_2,f^iq_2)|\ i\in\mathbb{N}\}\}$, then for any $\delta>0$, any $0<\varepsilon<\zeta$ and any $\theta\in[0,1]$, there exists $x_1,x_2\in X$ and $N\in\mathbb{N}$ such that for any $n>N$,
\begin{description}
\item[(a)] $\mathcal E_{n}(x_1)\in B(\theta\mu_1+(1-\theta)\mu_2,\varepsilon+\delta)\ \mathrm{and}\ \mathcal E_{n}(x_2)\in B(\theta\mu_1+(1-\theta)\mu_2,\varepsilon+\delta);$

\item[(b)] $\frac{|\{0\leq i\leq n-1|d(f^ix_1,f^ix_2)<\zeta-\varepsilon\}|}{n}<\delta$.
\end{description}
\end{Lem}

\begin{proof}
We just proof this lemma for $\theta$ is rational. Then, the lemma naturally holds for any $\theta\in[0,1]$ by the denseness of rational numbers. For any fixed $\delta>0$, $0<\varepsilon<\zeta$ and $\frac{\theta}{1-\theta}=\frac{s}{t}$, where $s,t\in\mathbb{N}^+$, we can obtain an $M_1$ such that $\mathcal E_{n}(p_i)\in B(\mu_i,\varepsilon/2)$ and $\mathcal E_{n}(q_i)\in B(\mu_i,\varepsilon/2), i=\{1,2\}$ hold for any $n\geq M_1$. We choose $M,r\in\mathbb{N}^+$ such that
\begin{equation}\label{pqlimit-1}
M>\max\{M_1,\frac{4K_{\varepsilon/2}}{\delta}\},
\end{equation}
\begin{equation}\label{pqlimit-2}
r>\frac{4}{\delta}.
\end{equation}
For any $k\geq 1$, by specification property, we can obtain an $x_1^k$ such that for any $j\in[0,k-1],\ i\in[0,s-1]$, $x_1^k$ $\varepsilon/2$-$traces$ $p_1$ on $[j(s+t)(M+K_{\varepsilon/2})+i(M+K_{\varepsilon/2}),j(s+t)(M+K_{\varepsilon/2})+(i+1)M+iK_{\varepsilon/2}]$ and for any $j\in[0,k-1],\ i\in[s,s+t-1]$, $x_1^k$ $\varepsilon/2$-$traces$ $p_2$ on $[j(s+t)(M+K_{\varepsilon/2})+i(M+K_{\varepsilon/2}),j(s+t)(M+K_{\varepsilon/2})+(i+1)M+iK_{\varepsilon/2}]$. Also we can obtain an $x_2^k$ such that for any $j\in[0,k-1],\ i\in[0,s-1]$, $x_2^k$ $\varepsilon/2$-$traces$ $q_1$ on $[j(s+t)(M+K_{\varepsilon/2})+i(M+K_{\varepsilon/2}),j(s+t)(M+K_{\varepsilon/2})+(i+1)M+iK_{\varepsilon/2}]$ and for any $j\in[0,k-1],\ i\in[s,s+t-1]$, $x_2^k$ $\varepsilon/2$-$traces$ $q_2$ on $[j(s+t)(M+K_{\varepsilon/2})+i(M+K_{\varepsilon/2}),j(s+t)(M+K_{\varepsilon/2})+(i+1)M+iK_{\varepsilon/2}]$. We can assume that(take subsequence if necessary) $x_1=\lim_{k\to\infty}x_1^k,\ x_2=\lim_{k\to\infty}x_2^k$. By the continuity of $f$, we have for any $j\in\mathbb{N},\ i\in[0,s-1]$, $x_1$ $\varepsilon/2$-$traces$ $p_1$ on $[j(s+t)(M+K_{\varepsilon/2})+i(M+K_{\varepsilon/2}),j(s+t)(M+K_{\varepsilon/2})+(i+1)M+iK_{\varepsilon/2}]$ and for any $j\in\mathbb{N},\ i\in[s,s+t-1]$, $x_1$ $\varepsilon/2$-$traces$ $p_2$ on $[j(s+t)(M+K_{\varepsilon/2})+i(M+K_{\varepsilon/2}),j(s+t)(M+K_{\varepsilon/2})+(i+1)M+iK_{\varepsilon/2}]$. Similarly, for any $j\in\mathbb{N},\ i\in[0,s-1]$, $x_2$ $\varepsilon/2$-$traces$ $q_1$ on $[j(s+t)(M+K_{\varepsilon/2})+i(M+K_{\varepsilon/2}),j(s+t)(M+K_{\varepsilon/2})+(i+1)M+iK_{\varepsilon/2}]$ and for any $j\in\mathbb{N},\ i\in[s,s+t-1]$, $x_2$ $\varepsilon/2$-$traces$ $q_2$ on $[j(s+t)(M+K_{\varepsilon/2})+i(M+K_{\varepsilon/2}),j(s+t)(M+K_{\varepsilon/2})+(i+1)M+iK_{\varepsilon/2}]$. Set $N:=r(s+t)(M+K_{\varepsilon/2})$, we will show that such $N$ and $x_1,x_2$ satisfy $(\mathbf{a})$ and $(\mathbf{b})$. For any $n>N$, $n$ lies in $[k(s+t)(M+K_{\varepsilon/2}),(k+1)(s+t)(M+K_{\varepsilon/2})]$ for some $k\geq r$. By (\ref{pqlimit-2}) and Lemma \ref{measure distance}, we have
\begin{equation}\label{delta/2}
d(\mathcal E_{n}(x_1),\mathcal E_{k(s+t)(M+K_{\varepsilon/2})}(x_1))<\frac{\delta}{2};\ \ \ d(\mathcal E_{n}(x_2),\mathcal E_{k(s+t)(M+K_{\varepsilon/2})}(x_2))<\frac{\delta}{2}
\end{equation}
Note that for any $j\in\mathbb{N},\ i\in[0,s-1]$, $x_1$ $\varepsilon/2$-$traces$ $p_1$ on $[j(s+t)(M+K_{\varepsilon/2})+i(M+K_{\varepsilon/2}),j(s+t)(M+K_{\varepsilon/2})+(i+1)M+iK_{\varepsilon/2}]$ and for any $j\in\mathbb{N},\ i\in[s,s+t-1]$, $x_1$ $\varepsilon/2$-$traces$ $p_2$ on $[j(s+t)(M+K_{\varepsilon/2})+i(M+K_{\varepsilon/2}),j(s+t)(M+K_{\varepsilon/2})+(i+1)M+iK_{\varepsilon/2}]$. We have
\begin{align*}
& d(\mathcal E_{k(s+t)(M+K_{\varepsilon/2})}(x_1),\theta\mathcal E_{M}(p_1)+(1-\theta)\mathcal E_{M}(p_2))\\
\leq & d(\sum_{i=1}^k\frac{1}{k}\mathcal E_{(s+t)(M+K_{\varepsilon/2})}(f^{(i-1)(s+t)(M+K_{\varepsilon/2})}x_1),\theta\mathcal E_{M}(p_1)+(1-\theta)\mathcal E_{M}(p_2))\\
\leq & \frac{1}{k}\sum_{i=1}^kd(\mathcal E_{(s+t)(M+K_{\varepsilon/2})}(f^{(i-1)(s+t)(M+K_{\varepsilon/2})}x_1),\theta\mathcal E_{M}(p_1)+(1-\theta)\mathcal E_{M}(p_2))\\
\leq & \frac{1}{k}\sum_{i=1}^k[d(\frac{s}{s+t}\mathcal E_{s(M+K_{\varepsilon/2})}(f^{(i-1)(s+t)(M+K_{\varepsilon/2})}x_1),\theta\mathcal E_{M}(p_1))\\
& + d(\frac{t}{s+t}\mathcal E_{t(M+K_{\varepsilon/2})}(f^{[(i-1)(s+t)+s](M+K_{\varepsilon/2})}x_1),(1-\theta)\mathcal E_{M}(p_2))]\\
< & \frac{1}{k}\sum_{i=1}^k[\theta(\varepsilon/2+\delta/2)+(1-\theta)(\varepsilon/2+\delta/2)]\\
= & \varepsilon/2+\delta/2.
\end{align*}
Combining with (\ref{delta/2}) and $\mathcal E_{M}(p_i)\in B(\mu_i,\varepsilon/2)$, we have $d(\mathcal E_{n}(x_1),\theta\mu_1+(1-\theta\mu_2))<\varepsilon+\delta$. Similarly, we can prove $d(\mathcal E_{n}(x_2),\theta\mu_1+(1-\theta\mu_2))<\varepsilon+\delta$. Hence $(\mathbf{a})$ holds. Note that $\zeta=\mathrm{min}\{\inf\{d(f^ip_1,f^iq_1)|\ i\in\mathbb{N}\},\inf\{d(f^ip_2,f^iq_2)|\ i\in\mathbb{N}\}\}$, then we have
$$\frac{|\{i|d(f^ix_1,f^ix_2)<\zeta-\varepsilon\}|}{n}<\frac{1}{k}+\frac{K_{\varepsilon/2}}{M}<\delta.$$
Hence $(\mathbf{b})$ holds.
\end{proof}

\subsection{Proof of Theorem \ref{maintheorem-DC1inSaturated}}  
We assume that $(p_1,q_1)$, $(p_2,q_2)$ is the distal pair of $G_{\mu_1}$, $G_{\mu_2}$ respectively and $\mathrm{min}\{\inf\{d(f^ip_1,f^iq_1)|\ i\in\mathbb{N}\},\inf\{d(f^ip_2,f^iq_2)|\ i\in\mathbb{N}\}\}=\zeta>0$. For any non-empty open set $U$, we can fix an $\varepsilon>0$ and a transitive point $z\in U$ such that $\overline{B(z,\varepsilon)}\subseteq U$ since transitive points are dense for system with specification property. Let $\varepsilon_i=\varepsilon/2^i$, $K_i=K_{\varepsilon_i}$$($cf.definition of specification property$)$. Let $\delta_1<1,\ \delta_i=\delta_{i-1}/2$. By \cite[Page 944]{PS2},  there exists a sequence $\{\alpha_1,\alpha_2,\cdots\}\subseteq K$ such that $$\overline{\{\alpha_j:j\in\mathbb{N}^+,j>n\}}=K,\ \forall n\in\mathbb{N}.$$ By Lemma \ref{lemma-mu}, for any $s\in\mathbb{N}^+$, we can obtain $x_1^{\varepsilon_s,\delta_s}$, $x_2^{\varepsilon_s,\delta_s}$ and $N^{\varepsilon_s,\delta_s}$ such that for any $n\geq N^{\varepsilon_s,\delta_s}$
\begin{equation}\label{result-of-lemma-pq1}
\mathcal E_n(x_1^{\varepsilon_s,\delta_s})\in B(\mu,\varepsilon_s+\delta_s),\ \mathcal E_n(x_2^{\varepsilon_s,\delta_s})\in B(\mu,\varepsilon_s+\delta_s),
\end{equation}
\begin{equation}\label{result-of-lemma-pq2}
\frac{|\{i\in[0,n-1]|d(f^ix_1^{\varepsilon_s,\delta_s},f^ix_2^{\varepsilon_s,\delta_s})<\zeta-\varepsilon\}|}{n}<\delta_s.
\end{equation}
Also, for any $s\in\mathbb{N}^+$, we can obtain an $M_{\varepsilon_s}^\mu$ such that the result of Lemma \ref{lemma-KKKK} holds.
Now, giving an $\xi=(\xi_1,\xi_2,\cdots)\in\{1,2\}^\infty$, we construct the $x_\xi$ inductively.

{\bf Step 1: $construct\ x_{\xi_1}$.} 
We fix $T_1=2K_1$. By Lemma \ref{lemma-KKKK}, for a large enough $M_1>M_{\varepsilon_1}^\mu$ satisfying
\begin{equation}
\delta_1M_1>\mathrm{max}\{T_1+2K_1,\ N^{\varepsilon_1,\delta_1}\}.
\end{equation}
we can obtain an $x_{\varepsilon_1}^{\alpha_1}$ and $t_2^{\varepsilon_1,\alpha_1}>t_1^{\varepsilon_1,\alpha_1}>M_1$ such that
\begin{equation}
\left\{
             \begin{array}{lr}
             \mathcal E_{n}(x_{\varepsilon_1}^{\alpha_1})\in B(\mu,\varepsilon_1),\ \forall n\in[M_{\varepsilon_1}^\mu,M_1];\\
             \mathcal E_{n}(x_{\varepsilon_1}^{\alpha_1})\in B(K,\varepsilon_1),\ \forall n\in[M_1,t_1^{\varepsilon_1,\alpha_1}];\\
             \mathcal E_{t_1^{\varepsilon_1,\alpha_1}}(x_{\varepsilon_1}^{\alpha_1})\in B(\alpha_1,\varepsilon_1);\\
             \mathcal E_{n}(x_{\varepsilon_1}^{\alpha_1})\in B(K,\varepsilon_1),\ \forall n\in[t_1^{\varepsilon_1,\alpha_1},t_2^{\varepsilon_1,\alpha_1}];\\
             \mathcal E_{t_2^{\varepsilon_1,\alpha_1}}(x_{\varepsilon_1}^{\alpha_1})\in B(\mu,\varepsilon_1).
             \end{array}
\right.
\end{equation}
Set $T_{1\rightarrow 2}=T_1+t_1^{\varepsilon_1,\alpha_1}$, $T_2=T_1+t_2^{\varepsilon_1,\alpha_1}$, $T_3=T_2+2K_1$, $T_4$ large enough such that
\begin{equation}
\delta_1T_4>\mathrm{max}\{T_3+2K_2,\ M_{\varepsilon_2}^\mu\},\ T_4-T_3>N^{\varepsilon_1,\delta_1}.
\end{equation}
By specification property, we can obtain an $x_{\xi_1}$ $\varepsilon_1$-$traces$ $z,x_{\varepsilon_1}^{\alpha_1},x_{\xi_1}^{\varepsilon_1,\delta_1}$ on $[0,0],[T_1,T_2],[T_3,T_4]$ respectively.

{\bf Step k: $construct\ x_{\xi_1\cdots\xi_k}.$}  If $x_{\xi_1\cdots\xi_{k-1}}$, $\{T_i\}_{i=1}^{2k(k-1)}$ and $\{T_{4i-3\rightarrow 4i-2}\}_{i=1}^{\frac{k(k-1)}{2}}$ have been defined, we construct $x_{\xi_1\cdots\xi_k}$ in the following way. For any $i\in\{1,2,\cdots,k\}$, let $T_{2k(k-1)+4i-2}$ and $T_{2k(k-1)+4i}$ be indefinite; $T_{2k(k-1)+4i-3}=T_{2k(k-1)+4i-4}+2K_k$ and $T_{2k(k-1)+4i-1}=T_{2k(k-1)+4i-2}+2K_k$. By Lemma \ref{lemma-KKKK}, for a large enough $M_{\frac{k(k-1)}{2}+i}>M_{\varepsilon_k}^\mu$ satisfying
\begin{equation}\label{guji-lemma1-1}
\delta_kM_{\frac{k(k-1)}{2}+i}>\mathrm{max}\{T_{2k(k-1)+4i-3}+2K_k,\ N^{\varepsilon_k,\delta_k}\}.
\end{equation}
we can obtain an $x_{\varepsilon_k}^{\alpha_i}$ and $t_2^{\varepsilon_k,\alpha_i}>t_1^{\varepsilon_k,\alpha_i}>M_{\frac{k(k-1)}{2}+i}$ such that
\begin{equation}\label{lemma-K-5}
\left\{
             \begin{array}{lr}
             \mathcal E_{n}(x_{\varepsilon_k}^{\alpha_i})\in B(\mu,\varepsilon_k),\ \forall n\in[M_{\varepsilon_k}^\mu,M_{\frac{k(k-1)}{2}+i}];\\
             \mathcal E_{n}(x_{\varepsilon_k}^{\alpha_i})\in B(K,\varepsilon_k),\ \forall n\in[M_{\frac{k(k-1)}{2}+i},t_1^{\varepsilon_k,\alpha_i}];\\
             \mathcal E_{t_1^{\varepsilon_k,\alpha_i}}(x_{\varepsilon_k}^{\alpha_i})\in B(\alpha_i,\varepsilon_k);\\
             \mathcal E_{n}(x_{\varepsilon_k}^{\alpha_i})\in B(K,\varepsilon_k),\ \forall n\in[t_1^{\varepsilon_k,\alpha_i},t_2^{\varepsilon_k,\alpha_i}];\\
             \mathcal E_{t_2^{\varepsilon_k,\alpha_i}}(x_{\varepsilon_k}^{\alpha_i})\in B(\mu,\varepsilon_k).
             \end{array}
\right.
\end{equation}
Set $T_{2k(k-1)+4i-3\rightarrow 2k(k-1)+4i-2}=T_{2k(k-1)+4i-3}+t_1^{\varepsilon_k,\alpha_i}$, $T_{2k(k-1)+4i-2}=T_{2k(k-1)+4i-3}+t_2^{\varepsilon_k,\alpha_i}$. If $i<k$, we select $T_{2k(k-1)+4i}$ is large enough such that
\begin{equation}\label{mu-next-step}
\delta_kT_{2k(k-1)+4i}>\mathrm{max}\{T_{2k(k-1)+4i-1}+2K_k,\ M_{\varepsilon_k}^\mu\},
\end{equation}
\begin{equation}
T_{2k(k-1)+4i}-T_{2k(k-1)+4i-1}>N^{\varepsilon_k,\delta_k}.
\end{equation}
If $i=k$, $T_{2k(k-1)+4i}$ is large enough such that
\begin{equation}\label{lemma-4i}
\delta_kT_{2k(k-1)+4i}>\mathrm{max}\{T_{2k(k-1)+4i-1}+2K_{k+1},\ M_{\varepsilon_{k+1}}^\mu\},
\end{equation}
\begin{equation}
T_{2k(k-1)+4i}-T_{2k(k-1)+4i-1}>N^{\varepsilon_k,\delta_k}.
\end{equation}
Hence we have defined the $T_{2(k-1)k+1},\cdots,T_{2k(k+1)}$ and $T_{2k(k-1)+4i-3\rightarrow 2k(k-1)+4i-2}\ \forall i\in[1,k]$. By specification property, we can obtain an $x_{\xi_1\cdots\xi_k}$ $\varepsilon_k$-$traces$ $x_{\xi_1\cdots\xi_{k-1}},f^{k-1}z,x_{\varepsilon_k}^{\alpha_1},\\
x_{\xi_1}^{\varepsilon_k,\delta_k},x_{\varepsilon_k}^{\alpha_2},x_{\xi_2}^{\varepsilon_k,\delta_k},\cdots,x_{\varepsilon_k}^{\alpha_k},x_{\xi_k}^{\varepsilon_k,\delta_k}$ on $[0,T_{2k(k-1)}],[T_{2k(k-1)}+K_k,T_{2k(k-1)}+K_k],[T_{2k(k-1)+1},$\\
$T_{2k(k-1)+2}],\cdots,[T_{2k(k-1)+4k-1},T_{2k(k-1)+4k}]$ respectively. Obviously, $d(x_{\xi_1\cdots\xi_{k-1}},x_{\xi_1\cdots\xi_k})<\varepsilon_k$, so $\{x_{\xi_1\cdots\xi_k}\}_{k=1}^{\infty}$ is a cauchy sequence in $\overline{B(z,\varepsilon)}$ since $\sum_{i=k}^{+\infty}\varepsilon_i\leq 2\varepsilon_k$. Denote the accumulation point of $\{x_{\xi_1\cdots\xi_k}\}_{k=1}^{\infty}$ by $x_\xi$, and it is easy to verify that $x_\xi$ $2\varepsilon_k$-$traces$ $f^{k-1}z,x_{\varepsilon_k}^{\alpha_1},\\
x_{\xi_1}^{\varepsilon_k,\delta_k},x_{\varepsilon_k}^{\alpha_2},x_{\xi_2}^{\varepsilon_k,\delta_k},\cdots,x_{\varepsilon_k}^{\alpha_k},x_{\xi_k}^{\varepsilon_k,\delta_k}$ on $[T_{2k(k-1)}+K_k,T_{2k(k-1)}+K_k],[T_{2k(k-1)+1},T_{2k(k-1)+2}],\cdots,$\\
$[T_{2k(k-1)+4k-1},T_{2k(k-1)+4k}]$ respectively since $\sum_{i=k}^{+\infty}\varepsilon_i\leq 2\varepsilon_k$. Note that $orb(x_\xi,f)$ has a subsequence which shadows the orbit of the transitive point $z$ closer and closer, so we can conclude that $x_\xi$ is also a transitive point. Fix $\xi,\eta\in\{1,2\}^\infty$, we claim that $x_\xi\neq x_\eta$ and $x_\xi,x_\eta$ is a DC1-scrambled pair if $\xi\neq \eta$. Suppose $\xi_s\neq \eta_s(\mathrm{implied\ by\ }\xi\neq \eta)$, then for any $k\geq s$ $x_\xi$ $2\varepsilon_k$-$traces$ $x_{\xi_s}^{\varepsilon_k,\delta_k}$ on [$T_{2(k-1)k+4s-1}$,$T_{2(k-1)k+4s}$] and $x_\eta$ $2\varepsilon_k$-$traces$ $x_{\eta_s}^{\varepsilon_k,\delta_k}$ on [$T_{2(k-1)k+4s-1}$,$T_{2(k-1)k+4s}$]. For any fixed $\kappa<\zeta$, we can get an $I_\kappa>s$ such that $\zeta-\kappa>5\varepsilon_{I_\kappa}$. Note that (\ref{result-of-lemma-pq2}), we have
$$\frac{|\{i\in[T_{2k(k-1)+4s-1},T_{2k(k-1)+4s}]|d(f^ix_{\xi_s}^{\varepsilon_k,\delta_k},f^ix_{\eta_s}^{\varepsilon_k,\delta_k})<\zeta-\varepsilon_k\}|}{T_{2k(k-1)+4s}-T_{2k(k-1)+4s-1}+1}<\delta_k<1$$
holds for any $k\geq I_\kappa$. So we have
$$\frac{|\{i\in[T_{2k(k-1)+4s-1},T_{2k(k-1)+4s}]|d(f^ix_{\xi},f^ix_{\eta})<\zeta-5\varepsilon_k\}|}{T_{2k(k-1)+4s}-T_{2k(k-1)+4s-1}+1}<\delta_k<1$$
holds for any $k\geq I_\kappa$, which implies for any $k\geq I_\kappa$, $\exists t\in [T_{2(k-1)k+4s-1},T_{2(k-1)k+4s}]$ such that $d(f^tx_{\xi},f^tx_{\eta})\geq\zeta-5\varepsilon_k>\kappa$. So $x_\xi\neq x_\eta$ and $\{x_\xi\}_{\xi\in \{1,2\}^\infty}$$($denote by $S)$ is an uncountable set. Meanwhile,
\begin{align*}
& \liminf_{n\to \infty}\frac{1}{n}|\{j\in [0,n-1]:\ d(f^jx_\xi,f^jx_\eta)<\kappa)\}|\\
\le & \liminf_{k\geq I_\kappa,\ k\to \infty}\frac{1}{T_{2(k-1)k+4s}}|\{j\in [0,T_{2(k-1)k+4s}-1]:\ d(f^jx_\xi,f^jx_\eta)<\kappa)\}|\\
\le & \liminf_{k\geq I_\kappa,\ k\to \infty}\frac{T_{2(k-1)k+4s-1}}{T_{2(k-1)k+4s}}+\delta_k\\
\le & \liminf_{k\geq I_\kappa,\ k\to \infty}2\delta_k
= 0.
\end{align*}
On the other hand, For any fixed $t>0$, we can choose $k_t\in\mathbb{N}$ large enough such that $4\varepsilon_k<t$ holds for any $k\geq k_t$. Note that $x_\xi$ and $x_\eta$ are both $2\varepsilon_k$-$tracea$ $x_{\varepsilon_k}^{\alpha_1}$ on [$T_{2(k-1)k+1}$,$T_{2(k-1)k+2}$]. So
\begin{align*}
&\limsup_{n\to \infty}\frac{1}{n}|\{j\in [0,n-1]:\ d(f^ix_\xi,f^ix_\eta)<t)\}|\\
\ge & \limsup_{n\to \infty}\frac{1}{n}|\{j\in [0,n-1]:\ d(f^jx_\xi,f^jx_\eta)<4\varepsilon_{k_t})\}|\\
\ge & \limsup_{k\geq k_t,\ k\to \infty}\frac{1}{T_{2(k-1)k+2}}|\{j\in [0,T_{2(k-1)k+2}-1]:\ d(f^jx_\xi,f^jx_\eta))<4\varepsilon_k\}|\\
\ge & \limsup_{k\geq k_t,\ k\to \infty}(1-\frac{T_{2(k-1)k+1}}{T_{2(k-1)k+2}})\\
\ge & \limsup_{k\geq k_t,\ k\to \infty}(1-\delta_k)\\
= & 1.
\end{align*}
So far, we have proved that $S=\{x_\xi\}_{\xi\in \{1,2\}^\infty}\subseteq \overline{B(z,\varepsilon)}\subseteq U$ is an   uncountable DC1-scrambled set. To finish this proof,  we need to check that $V_f(x_\xi)=K$ for any $\xi\in\{1,2\}^\infty$. On one hand, for any fixed $s\in\mathbb{N}^+$, when $k\geq s$,
%\begin{align*}
%& d(\mathcal{E}_{T_{2(k-1)k+4s-3\rightarrow 2(k-1)k+4s-2}}(x_\xi),\ \alpha_s)\\
%= & d(\frac{T_{2(k-1)k+4s-3}}{T_{2(k-1)k+4s-3\rightarrow 2(k-1)k+4s-2}}\mathcal{E}_{T_{2(k-1)k+4s-3}}(x_\xi)+\frac{T_{2(k-1)k+4s-3\rightarrow 2(k-1)k+4s-2}-T_{2(k-1)k+4s-3}}{T_{2(k-1)k+4s-3\rightarrow 2(k-1)k+4s-2}}\\
%& \mathcal{E}_{T_{2(k-1)k+4s-3\rightarrow 2(k-1)k+4s-2}-T_{2(k-1)k+4s-3}}(f^{T_{2(k-1)k+4s-3}}x_\xi),\ \alpha_s)\\
%\end{align*}
note (\ref{guji-lemma1-1}), $T_{2(k-1)k+4s-3\rightarrow 2(k-1)k+4s-2}-T_{2(k-1)k+4s-3}>M_{\frac{k(k-1)}{2}+s}$, and $x_\xi$ $2\varepsilon_k$-$traces$ $x_{\varepsilon_k}^{\alpha_s}$ on [$T_{2(k-1)k+4s-3}$,$T_{2(k-1)k+4s-3\rightarrow 2(k-1)k+4s-2}$], so we have
\begin{align*}
& d(\mathcal{E}_{T_{2(k-1)k+4s-3\rightarrow 2(k-1)k+4s-2}}(x_\xi),\ \alpha_s)\\
\le & d(\mathcal{E}_{T_{2(k-1)k+4s-3\rightarrow 2(k-1)k+4s-2}-T_{2(k-1)k+4s-3}}(f^{T_{2(k-1)k+4s-3}}x_\xi),\ \alpha_s)+2\delta_k\\
\le & d(\mathcal{E}_{T_{2(k-1)k+4s-2}-T_{2(k-1)k+4s-3}}(x_{\varepsilon_k}^{\alpha_s}),\ \alpha_s)+2\varepsilon_k+2\delta_k\\
\le & \varepsilon_k+2\varepsilon_k+2\delta_k\\
= & 3\varepsilon_k+2\delta_k
\end{align*}
by Lemma \ref{measure distance}. Let $k\to\infty,$ we have $\alpha_s\in V_f(x_\xi)$ for any $s\in\mathbb{N}^+$, which implies $K\subseteq V_f(x_\xi)$.

On the other hand, for any fixed $n\in\mathbb{N}^*$, we consider $\mathcal{E}_{n}(x_\xi)$. Obviously, there is a $k\in\mathbb{N}$ such that $n\in[T_{2(k-1)k+1}, T_{2k(k+1)}+2K_{k+1}]$. If $n$ lies in $[T_{2(k-1)k+4s-3}, T_{2(k-1)k+4s-2}+2K_k]$ for certain $s\in\{2,3,\cdots,k\}$,
\begin{align*}
\mathcal{E}_{n}(x_\xi)=\ \ \ \ \ & \frac{T_{2(k-1)k+4s-3}}{n}\mathcal{E}_{T_{2(k-1)k+4s-3}}(x_\xi)\\
&+\frac{n-T_{2(k-1)k+4s-3}}{n}\mathcal{E}_{n-T_{2(k-1)k+4s-3}}(f^{T_{2(k-1)k+4s-3}}x_\xi).
\end{align*}
Notice that $T_{2(k-1)k+4s-3}=T_{2(k-1)k+4(s-1)}+2K_k$, $x_\xi\ 2\varepsilon_k$-$traces$ $x_{\xi_s}^{\varepsilon_k,\delta_k}$ on $[T_{2(k-1)k+4(s-1)-1},\\
T_{2(k-1)k+4(s-1)}]$ and (\ref{result-of-lemma-pq1}),(\ref{mu-next-step}), so by Lemma \ref{measure distance}, we have
\begin{align*}
d(\mathcal{E}_{T_{2(k-1)k+4s-3}}(x_\xi),\mu) & <d(\mathcal{E}_{T_{2(k-1)k+4(s-1)}-T_{2(k-1)k+4(s-1)-1}}(f^{T_{2(k-1)k+4(s-1)-1}}x_\xi),\ \mu)+2\delta_k\\
& <d(\mathcal{E}_{T_{2(k-1)k+4(s-1)}-T_{2(k-1)k+4(s-1)-1}}(x_{\xi_s}^{\varepsilon_k,\delta_k}),\ \mu)+2\varepsilon_k+2\delta_k\\
& <\varepsilon_k+\delta_k+2\varepsilon_k+2\delta_k,
\end{align*}
i.e.,
\begin{equation}\label{easy-guji}
d(\mathcal{E}_{T_{2(k-1)k+4s-3}}(x_\xi),\mu)<3\varepsilon_k+3\delta_k.
\end{equation}
If $n\in[T_{2(k-1)k+4s-3},T_{2(k-1)k+4s-3}+M_{\varepsilon_k}^\mu]$, note that (\ref{guji-lemma1-1}) and $M_{\frac{2k(k-1)}{2}+s}>M_{\varepsilon_k}^\mu$, then we have $d(\mathcal{E}_{n}(x_\xi),\mathcal{E}_{T_{2(k-1)k+4s-3}}(x_\xi))<2\delta_k$ by Lemma \ref{measure distance}.
So,
\begin{equation}
d(\mathcal{E}_{n}(x_\xi),\mu)<2\delta_k+3\varepsilon_k+3\delta_k=3\varepsilon_k+5\delta_k.
\end{equation}
If $n\in[T_{2(k-1)k+4s-3}+M_{\varepsilon_k}^\mu,T_{2(k-1)k+4s-3}+M_{\frac{2k(k-1)}{2}+s}]$, by (\ref{lemma-K-5}), one has
\begin{align*}
d(\mathcal{E}_{n-T_{2(k-1)k+4s-3}}(f^{T_{2(k-1)k+4s-3}}x_\xi),\mu) & <d(\mathcal{E}_{n-T_{2(k-1)k+4s-3}}(x_{\varepsilon_k}^{\alpha_s}),\mu)+2\varepsilon_k\\
& < \varepsilon_k+2\varepsilon_k\\
& = 3\varepsilon_k.
\end{align*}
Combine with (\ref{easy-guji}), one has
\begin{equation}
d(\mathcal{E}_{n}(x_\xi),\mu)< 3\varepsilon_k+3\delta_k.
\end{equation}
If $n\in[T_{2(k-1)k+4s-3}+M_{\frac{2k(k-1)}{2}+s},T_{2(k-1)k+4s-2}+2K_k]$, by (\ref{guji-lemma1-1}) and Lemma \ref{measure distance}, we have
\begin{equation}
d(\mathcal{E}_{n}(x_\xi), \mathcal{E}_{n-T_{2(k-1)k+4s-3}}(f^{T_{2(k-1)k+4s-3}}x_\xi))<2\delta_k.
\end{equation}
Then $\mathcal{E}_{n}(x_\xi)\in B(K,\varepsilon_k+2\delta_k)$ by (\ref{lemma-K-5}). So $\mathcal{E}_{n}(x_\xi)\subseteq B(K,3\varepsilon_k+5\delta_k)$ when $n\in[T_{2(k-1)k+4s-3}, T_{2(k-1)k+4s-2}+2K_k]$. In other situations of the interval where $n$ lies, we can also prove $\mathcal{E}_{n}(x_\xi)\subseteq B(K,3\varepsilon_k+5\delta_k)$ with a little modification of the method above. When $n\to\infty$, forcing $k\to\infty$, $B(K,3\varepsilon_k+5\delta_k)\to K$, hence we have $\mathcal{E}_{n}(x_\xi)=K$.
\qed

\begin{Rem}
Theorem \ref{maintheorem-DC1inSaturated} just states the situation where $K$ contains a measure $\mu$ which is the convex combination of two measures. Actually, with little modification, Theorem \ref{maintheorem-DC1inSaturated} also holds for any $K\subseteq \mathcal M_f(X)$ if $K$ contains a measure $\mu$ which is the convex combination of finite measures. Here we omit it.
\end{Rem}

\section{Proof of Theorems \ref{maintheorem-combination}, \ref{maintheorem-recurrent}  and \ref{Th-E-statisticallanguage}}\label{section-mainproof-combination}
Similar as  \cite{T16}, \cite{HTW} and \cite{DongTian2017}, we also deal with many  refined recurrent levels which will used not only to prove Theorems \ref{maintheorem-combination} and \ref{maintheorem-recurrent} but also to show Theorem \ref{Th-E-statisticallanguage}. Now let us recall their definitions.  
%For any $\mu,\mu_1,\mu_2\in \mathcal M_f(X)$; $M,N\subseteq\mathcal M_f(X)$, we denote $S_\mu=supp(\mu)=\{x\in X|\ \mu(U)>0\ $for any neighborhood $U$ of $x\}$ the support of $\mu$ and $\mathrm{cov}\{M,N\}=\{\theta\lambda_1+(1-\theta)\lambda_2|\ \theta\in[0,1],\ \lambda_1\in M,\ \lambda_2\in N\}$, particularly $\mathrm{cov}\{\mu_1,\mu_2\}=\{\theta\mu_1+(1-\theta)\mu_2|\ \theta\in[0,1]\}$. Obviously, $\mathrm{cov}\{M,N\}\subseteq\mathcal M_f(X)$. 
  Given $x\in X$, let
$C_x=\overline{\bigcup_{m\in V_f(x)}S_m}.$ Let $BR^\#:=BR\setminus QW$,
\begin{align*}
W^\#\ \ &:=\ \ \{x\in BR^\#\ |\ S_\mu=C_x\ \mathrm{for\ every}\ \mu\in V_f(x)\},\\
V^\#\ \ &:=\ \ \{x\in BR^\#\ |\ \exists\mu\in V_f(x)\ \mathrm{such\ that}\ S_\mu=C_x\},\\
S^\#\ \ &:=\ \ \{x\in X|\ \cap_{\mu\in V_f(x)}S_\mu\neq\emptyset\}.
\end{align*}
Then we can divide $BR^\#$ into following several levels with different asymptotic behavior:
\begin{align*}
BR_1\ \ &:=\ \ W^\#,\\
BR_2\ \ &:=\ \ V^\#\cap S^\#,\ \ BR_3\ \ :=\ \ V^\#,\\
BR_4\ \ &:=\ \ V^\#\cup(BR^\#\cap S^\#),\ \ BR_5\ \ :=\ \ BR^\#.
\end{align*}
Immediately, $BR_1\subseteq BR_2\subseteq BR_3\subseteq BR_4\subseteq BR_5$.
Denote
\begin{align*}
W^*\ \ &:=\ \ \{x\in QW\ |\ S_\mu=C_x\ \mathrm{for\ every}\ \mu\in V_f(x)\},\\
V\ \ &:=\ \ \{x\in QW\ |\ \exists\mu\in V_f(x)\ \mathrm{such\ that}\ S_\mu=C_x\},\\
S\ \ &:=\ \ \{x\in X|\ \cap_{\mu\in V_f(x)}S_\mu\neq\emptyset\}.
\end{align*}
Later, we will see that $W^*=W$. Now we can divide $QW$ into following several levels with different asymptotic behavior:
\begin{align*}
QW_1\ \ &:=\ \ W^*,\\
QW_2\ \ &:=\ \ V\cap S,\ \ QW_3\ \ :=\ \ V,\\
QW_4\ \ &:=\ \ V\cup(QW\cap S),\ \ QW_5\ \ :=\ \ QW.
\end{align*} 
These levels are related the different statistical $\omega$-limit sets, see Section \ref{section-statistical}.

For a collection of subsets $Z_1,Z_2,\cdots,Z_k\subseteq X(k\geq 2)$, we denote GS$\{Z_1, 
Z_2,\cdots,Z_k\}=\{Z_{2}\setminus Z_1,Z_{3}\setminus Z_2,\cdots,Z_k\setminus Z_{k-1}\}$ the gap sets of the sequence.
\begin{Def}
 We say $\{Z_i\}_{i=1}^k$ has  uncountable DC1-scrambled gap with respect to $Y\,(\subseteq X)$ if $\,S\,\cap \,Y$ contains an   uncountable DC1-scrambled subset for any $S\in \mathrm{GS}\{Z_1,Z_2,\cdots,Z_k\}$.
\end{Def}

\begin{Thm}\label{Th-best-refinedversion}
Suppose that $(X,f)$ has specification property, $\varphi$ is a continuous function on $X$. % and $I_{\varphi}(f)\neq\emptyset$. 
Then 
\begin{description}
\item[(a)] If $(X,f)$ is not uniquely ergodic, $\{QW_1, QW_2, QW_3, QW_4, QW_5, BR_1, BR_2, BR_3,  BR_4, BR_5\}$ has  uncountable DC1-scrambled gap with respect to $ Trans$;

\item[(b)] If $I_{\varphi}\neq \emptyset$, then $\{QW_1, QW_2, QW_3, QW_4, QW_5, BR_1, BR_2, BR_3, BR_4, BR_5\}$ has  uncountable DC1-scrambled gap with respect to $ Trans\cap I_{\varphi}$;

\item[(c)] If $Int(L_\varphi)\neq\emptyset$, then for any $a\in Int(L_\varphi)$, $\{QW_1, QW_2, QW_3, QW_4, QW_5, BR_1, 
 BR_2, BR_3, BR_4, \\ BR_5\}$ has  uncountable DC1-scrambled gap with respect to $Trans\cap R_{\varphi}(a)$.

 \item[(d)]   If $(X,f)$ is not uniquely ergodic,  $\{QW_1, QW_2, QW_3, QW_4, QW_5, BR_1, 
 BR_2, BR_3,  BR_4, BR_5\}$ has  uncountable DC1-scrambled gap with respect to $ Trans\cap R_{\varphi}$.

\end{description}
\end{Thm}

\begin{Rem} If $a\in L_\varphi \setminus Int(L_\varphi),$ Theorem  \ref{Th-best-refinedversion} may be not true even for Li-Yorke chaotic. For example, if $(X,f)$ is full shift of two symbols (which satisfies specification), taking $orb(p,f), orb(q,f)$ to be two different periodic orbits with period $\geq 2$ and letting  $\varphi$ be a continuous function such that $\varphi|_{orb(p,f)}=0$, $\varphi|_{orb(q,f)}=1$ and for any $x\in X\setminus (orb(p,f)\cup orb(q,f)) $, $0<\varphi(x)<1$.  In this case $L_\phi=[0,1]$. Let $\mu_p,\mu_q$ denote the periodic measures supported on the orbit of $p,q$. It is not difficult to check that  $G_{\mu_p}\cap Trans\subseteq R_\phi(0)\cap Trans \subseteq BR_1$ and $G_{\mu_q}\cap Trans\subseteq R_\phi(1)\cap Trans \subseteq BR_1$ so that $\{QW_1, QW_2, QW_3, QW_4, QW_5\} $ and $\{  BR_1, BR_2, 
BR_3, BR_4, BR_5\}$ have empty gap with respect to $R_\phi(0)\cap Trans  $ and $R_\phi(1)\cap Trans  $.  So most cases can not have  any kind of chaotic behavior with respect to $R_\phi(0)\cap Trans  $ and $R_\phi(1)\cap Trans  $. By Theorem \ref{maintheorem-DC1inSaturated} $G_{\mu_p},G_{\mu_q} $ %$R_\phi(0) ,R_\phi(1) $ and 
   all contain uncountable DC1-scrambled subsets and so do $R_\phi(0)\cap Trans\cap BR_1, R_\phi(1)\cap Trans\cap BR_1$ .  However,  
%for which everyone is composed of one periodic orbit so that does not contain transitive points and Li-Yorke pair.
   $R_\phi(0)   $ and $R_\phi(1)   $ has zero topological entropy by (\ref{eq-1}). In particular, this implies that there exists an uncountable DC1-scrambled  set with zero topological entropy. 
\end{Rem}

Theorem \ref{Th-best-refinedversion} implies  Theorems  \ref{maintheorem-combination} and \ref{maintheorem-recurrent}   so that we only need to prove Theorem \ref{Th-best-refinedversion} .

\subsection{Distal Pair in Minimal Sets}%Technique Lemmas}
\label{section-mainlemma}

\begin{Prop}\label{distal-dense}
Suppose that $(X,f)$ has specification property, then
%then the measures whose generic points has distal pair is dense in $\mathcal M_f(X)$. Actually,
$$\{\mu\in \mathcal M_f(X)|\mu\ is\ ergodic,\ S_\mu\ is\ nondegenerate\ and\ minimal \}$$ is dense in $\mathcal M_f(X)$
and for any $\mu$ in such set, $G_\mu$ has distal pair.
\end{Prop}

To prove Proposition \ref{distal-dense}, we need some preliminaries. An infinite set $A=\{a_1<a_2<\cdots\}\subseteq \mathbb{N}$ is syndetic if there is an $N\in \mathbb{N}$ such that $a_{i+1}-a_i\leq N$ holds for any $i\in \mathbb{N}$. Denote $\mathcal{D}(A)=\mathrm{min}\{N\in\mathbb{N}\ |\ a_{i+1}-a_i\leq N\ \mathrm{holds\ for\ any}\ i\in \mathbb{N}\}$ and $\mathcal{F}_s=\{A\subseteq\mathbb{N}|\ A\ \mathrm{is\ syndetic}\}.$

\begin{Lem}\label{sydetic}
Given $(X,f)$, for any $p,q\in X$, if there is an $\varepsilon>0$ such that $\{i\ |\ d(f^ip,f^iq)>\varepsilon\}\in\mathcal{F}_s$. Then $p,q$ is distal.
\end{Lem}

\begin{proof}
Suppose $p,q,\varepsilon$ is fixed, $\mathcal{D}(\{i\ |\ d(f^ip,f^iq)>\varepsilon\})=M$. Obviously $f$ is uniform continuous since $f$ is continuous and $X$
is compact. So we can get $\eta_1$ such that for any $x,y\in X$, if $d(x,y)<\eta_1$, then $d(fx,fy)<\varepsilon$. By induction, we get $\eta_k$ such that for any $x,y\in X$, if $d(x,y)<\eta_k$, then $d(fx,fy)<\eta_{k-1}$, until $k=M$. Set $\eta=\mathrm{min}\{\varepsilon,\eta_1,\eta_2,\cdots,\eta_M\}$, we claim that $\liminf_{n\to\infty}d(f^np,f^nq)\geq\eta$. If not, there is an $n_0\in\mathbb{N}$ such that $d(f^{n_0}p,f^{n_0}q)<\eta$. By the discussion above, we have $d(f^{n_0+k}p,f^{n_0+k}q)<\varepsilon$ for any $k\in\{0,1,\cdots,M\}$, which conflicts with $\mathcal{D}(\{i\ |\ d(f^ip,f^iq)>\varepsilon\})=M.$
\end{proof}

\begin{Lem}\label{generic distal}
Given $(X,f)$. Suppose that $\mu\in\mathcal{M}_f^e(X)$, $S_\mu$ is nondegenerate and minimal. Then, there are two distinct points $p,q\in G_\mu$ such that $p,q$ is distal.
\end{Lem}

\begin{proof}
By the hypothesis, we can choose two distinct points $u,v\in S_\mu$. Denote $B_u,B_v$ the open neighborhood of $u,v$ respectively. Here we can assume $B_u\cap B_v=\emptyset$ and $d(B_u,B_v)=\zeta>0$ since $X$ is a metric space. Obviously $\mu(S_\mu)=1,\mu(B_u)>0,\mu(B_v)>0.$ Notice that $\mu$ is ergodic, so $\mu(G_\mu)=1$ and there exists an $M\in \mathbb{N}$ such that $\mu(B_u\cap f^{-M}B_v)>0$. So $\mu(B_u\cap f^{-M}B_v\cap G_\mu\cap S_\mu)>0$. Fix a $p\in B_u\cap f^{-M}B_v\cap G_\mu\cap S_\mu$, then $N(p,B_u\cap f^{-M}B_v)=\{a_1<a_2<\cdots\}\in \mathcal{F}_s$ since $p\in S_\mu$ is a minimal point and $B_u\cap f^{-M}B_v$ is an open neighborhood of $p$. Set $q=f^Mp$, for any $k\in \mathbb{N}$, we have $f^{a_k}q\in B_v$ since $f_{a_k}p\in B_u\cap f^{-M}B_v$. So $d(f^{a_k}p, f^{a_k}q)\geq\zeta>0$. Notice that $\{a_1,a_2,\cdots\}\in\mathcal{F}_s$, so $p,q$ is distal by lemma \ref{sydetic}. $p\in G_\mu\Rightarrow q\in G_\mu$.
\end{proof}

{\bf Proof of Proposition \ref{distal-dense}}
For system $(X,f)$ with specification, we have
$$\{\mu\in \mathcal M_f(X)|\mu\ is\ ergodic,\ S_\mu\ is\ minimal \}$$
is dense in $\mathcal M_f(X)$, which is a direct corollary of \cite[Theorem A]{HTW}. Here we claim that
$$\{\mu\in \mathcal M_f(X)|\mu\ is\ ergodic,\ S_\mu\ is\ nondegenerate\ and\ minimal \}$$
is also dense in $\mathcal M_f(X)$. If not, there will be a open set $U\subseteq \mathcal M_f(X)$ such that
$$\{\mu\in \mathcal M_f(X)|\mu\ is\ ergodic,\ S_\mu\ is\ degenerate\ and\ minimal \}$$
is dense in $U$, which implies that any measure in $U$ can be approximated by the Dirac measure concentrate on a fix point. i.e. for any $\mu\in U$, there is a sequence $\{x_i\}_{i=1}^{\infty}$ such that $\lim_{i\to\infty}\delta_{x_i}=\mu$. Without loss of generality, we can assume that $\lim_{i\to\infty}x_i=x$. Then for any continuous function $f$ on $X$, $$\int fd\mu=\lim_{i\to\infty}\int fd\delta_{x_i}=\lim_{i\to\infty}f(x_i)=f(x)=\int fd\delta_x.$$
So we have $\mu=\delta_x$, which means measures in $U$ are all Dirac measures, which conflict with Proposition \ref{specification-full-support-dense}. So the conflict and Lemma \ref{generic distal} end this proof.
\qed

\subsection{Proof of Theorem \ref{Th-best-refinedversion} }

For any $ \mu_1,\mu_2\in \mathcal M_f(X)$, we define $$\mathrm{cov}\{\mu_1,\mu_2\}=\{\theta\mu_1+(1-\theta)\mu_2|\ \theta\in[0,1]\}.$$ %Obviously, $\mathrm{cov}\{M,N\}\subseteq\mathcal M_f(X)$.

\bigskip

{\bf Proof of Item (a). }
By \cite[Lemma 3.4]{HTW}, we can take $\mu_1,\mu_2,\cdots$ satisfying Proposition \ref{distal-dense} and $\overline{\bigcup_{i=1}^{\infty} S_{\mu_i}}=X$. Then their support are naturally mutually disjoint and for any finite set $\Lambda\subseteq\mathbb{N}^+$, $\bigcup_{i\in \Lambda} S_{\mu_i}\neq X$ since $S_{\mu_i}$ is minimal. Let $\mu$ be a measure with full support and take $\nu_i=\frac{i-1}{i}\mu_1+\frac{1}{i}\mu_i,\ i\in\{1,2,\cdots\}$. then we have
\begin{align}\label{construct-measure-QW}
\lim_{i\to\infty}d(\nu_i,\mu_1)\ &=\ \lim_{i\to\infty}d(\frac{i-1}{i}\mu_1+\frac{1}{i}\mu_i,\frac{i-1}{i}\mu_1+\frac{1}{i}\mu_1)\notag,\\
&\leq\ \lim_{i\to\infty}\frac{1}{i}d(\mu_i,\mu_1),\\
&\leq\ \lim_{i\to\infty}\frac{1}{i}\notag,\\
&=\ 0.\notag
\end{align}
Here we consider $\bigcup_{i=1}^\infty\mathrm{cov}\{\nu_i,\nu_{i+1}\}$. By (\ref{construct-measure-QW}), it is easy to check that $\bigcup_{i=1}^\infty\mathrm{cov}\{\nu_i,\nu_{i+1}\}$ is connected and compact. One can observe that $S_\kappa\neq X$ for any $\kappa\in\bigcup_{i=1}^\infty\mathrm{cov}\{\nu_i,\nu_{i+1}\}$. Moreover, $\bigcap_{\kappa\in\bigcup_{i=1}^\infty\mathrm{cov}\{\nu_i,\nu_{i+1}\}}S_\kappa=S_{\mu_1}$ and $\overline{\bigcup_{\kappa\in\bigcup_{i=1}^\infty\mathrm{cov}\{\nu_i,\nu_{i+1}\}}S_\kappa}=X$. Let
\begin{align*}
K_1\ &:=\ \ \mathrm{cov}\{\mu_1,\mu\},\\
K_2\ &:=\ \ \mathrm{cov}\{\mu_1,\mu\}\cup\mathrm{cov}\{\mu_1,\mu_2\},\\
K_3\ &:=\ \ \bigcup_{i=1}^\infty\mathrm{cov}\{\nu_i,\nu_{i+1}\},\\
K_4\ &:=\ \ \bigcup_{i=1}^\infty\mathrm{cov}\{\nu_i,\nu_{i+1}\}\cup\mathrm{cov}\{\mu_1,\mu_2\},\\
K_5\ &:=\ \ \{ \mu_1\},\\
K_6\ &:=\ \ \mathrm{cov}\{\mu_1,\nu_2\},\\%\\
%\end{align*}
%\begin{align*}
K_7\ &:=\ \ \mathrm{cov}\{\mu_1,\mu_2\},\\
K_8\ &:=\ \ \mathrm{cov}\{\mu_1,\nu_2\}\cup\mathrm{cov}\{\mu_1,\nu_3\},\\
K_9\ &:=\ \ \mathrm{cov}\{\mu_1,\mu_2\}\cup\mathrm{cov}\{\mu_1,\mu_3\}.
\end{align*}
Using Theorem \ref{maintheorem-DC1inSaturated} on $K_i, i\in\{1,2,3,4,5,6,7,8,9\}$, for any open set $U$, there is an uncountable scramble set $S_i\subseteq G_{K_i}\cap U\cap Trans$. By Proposition \ref{Trans-BR} and Proposition \ref{prop2}(c), we have $S_i\subseteq Trans$ implies $S_i\subseteq QW$ for $i\in\{1,2,3,4\}$ and $S_i\subseteq BR^\#$ for $i\in\{5,6,7,8,9\}$. One can observe that $G_{K_i}\subseteq QW_{i+1}\setminus QW_{i}$ for any $i\in\{1,2,3,4\}$, $G_{K_5}\subseteq BR_1\setminus QW_5$, $G_{K_i}\subseteq BR_{i-4}\setminus BR_{i-5}$ for any $i\in\{6,7,8,9\}$. Then the proof is completed.
\qed
\begin{Rem}  Let $QR=\bigcup_{\mu\in \mathcal M_f(X)}G_\mu$. The points in $QR$ are called quasiregular points of $f$ in \cite{Sig}. 
Note that $K_5$ in the proof above is a single measure, so we can replace $BR_{1}$ by $BR_{1}\cap QR$ and the theorem still holds. The same situation will happen in the proof of item (c).
\end{Rem}

\bigskip
{\bf Proof of Item (b). }
If $I_{\varphi}(f)\neq\emptyset$, then there exist $\lambda_1,\lambda_2\in\mathcal M_f(X)$ such that $\int\varphi d\lambda_1\neq\int\varphi d\lambda_2$. Note that the measure satisfying Proposition \ref{distal-dense} and measures with full support are both dense in $\mathcal M_f(X)$. Then we can choose $\mu_1,\mu_2,\cdots$ satisfying Proposition \ref{distal-dense} and $\overline{\bigcup_{i=1}^{\infty} S_{\mu_i}}=X$ such that $\int\varphi d\mu_1\neq\int\varphi d\mu_2\neq\int\varphi d\mu_3\neq\int\varphi d\mu$. Take $\nu_i=\frac{i-1}{i}\mu_1+\frac{1}{i}\mu_i,\ i\in\{1,2,\cdots\}$. Let
%\begin{align*}
%K_1\ &:=\ \ \mathrm{cov}\{\mu_1,\nu_1\}\\
%K_2\ &:=\ \ \mathrm{cov}\{\mu_1,\mu_2\}\\
%K_3\ &:=\ \ \mathrm{cov}\{\mu_1,\nu_1\}\cup\mathrm{cov}\{\mu_1,\nu_2\}\\
%K_4\ &:=\ \ \mathrm{cov}\{\mu_1,\mu_2\}\cup\mathrm{cov}\{\mu_1,\mu_3\}
%\end{align*}
\begin{align*}
K_1\ &:=\ \ \mathrm{cov}\{\mu_1,\mu\},\\
K_2\ &:=\ \ \mathrm{cov}\{\mu_1,\mu\}\cup\mathrm{cov}\{\mu_1,\mu_2\},\\
K_3\ &:=\ \ \bigcup_{i=1}^\infty\mathrm{cov}\{\nu_i,\nu_{i+1}\},\\
K_4\ &:=\ \ \bigcup_{i=1}^\infty\mathrm{cov}\{\nu_i,\nu_{i+1}\}\cup\mathrm{cov}\{\mu_1,\mu_2\}, \\
K_5\ &:=\ \ \mathrm{cov}\{\nu_2,\frac{1}{3}\mu_1+\frac{2}{3}\mu_2\},\\
K_6\ &:=\ \ \mathrm{cov}\{\mu_1,\nu_2\},\\%\\
%\end{align*}
%\begin{align*}
K_7\ &:=\ \ \mathrm{cov}\{\mu_1,\mu_2\},\\
K_8\ &:=\ \ \mathrm{cov}\{\mu_1,\nu_2\}\cup\mathrm{cov}\{\mu_1,\nu_3\},\\
K_9\ &:=\ \ \mathrm{cov}\{\mu_1,\mu_2\}\cup\mathrm{cov}\{\mu_1,\mu_3\}.
\end{align*}
%Using Theorem \ref{maintheorem-DC1inSaturated} on $K_i,i\in\{1,2,3,4\}$, for any open set $U$, there is an uncountable scramble set $S_i\subseteq G_{K_i}\cap U\cap Trans$. By Proposition \ref{Trans-BR} and Proposition \ref{prop2}(c), we have $S_i\subseteq Trans$ implies $S_i\subseteq BR^\#$ since $\bigcup_{i=1}^3 S_{\mu_i}\neq X$. One can observe that $G_{K_i}\subseteq BR_{i+1}\setminus BR_{i},\ i\in\{1,2,3,4\}$ and $G_{K_i}\subseteq I_{\varphi}(f)$ since $\int\varphi d\mu_1\neq\int\varphi d\mu_2\neq\int\varphi d\mu_3$, so the proof is completed.
One can observe that $G_{K_i}\subseteq I_{\varphi}(f)$ for any $i\in\{1,2,3,4,5,6,7,8,9\}$. Based on the discussion in the proof of item (a), we complete the proof.
\qed

\bigskip

{\bf Proof of Item (c). }
If $Int(L_\varphi)\neq\emptyset$, then for any $a\in Int(L_\varphi)$, there exist $\lambda_1,\lambda_2\in\mathcal M_f(X)$ such that $\int\varphi d\lambda_1<a<\int\varphi d\lambda_2$. By \cite[Lemma 3.4]{HTW}, we can take $\mu_1,\mu_2,\cdots$ satisfying Proposition \ref{distal-dense} and $\overline{\bigcup_{i=1}^{\infty} S_{\mu_i}}=X$. We can assume that $\{i\in[1,+\infty)|\ \int\varphi d\mu_i>a\}$ and $\{i\in[1,+\infty)|\ \int\varphi d\mu_i<a\}$ are both infinite set since measures satisfying Proposition \ref{distal-dense} are dense in $\mathcal M_f(x)$. Set $\{i\in[1,+\infty)|\ \int\varphi d\mu_i>a\}=\{m_i\}_{i=1}^{\infty}$ and $\{i\in[1,+\infty)|\ \int\varphi d\mu_i<a\}=\{n_i\}_{i=1}^{\infty}$. In order to simplify the proof, we assume $\{i\in[1,+\infty)|\ \int\varphi d\mu_i=a\}=\emptyset$. Now, we can choose proper $\{\theta_i\}_{i=1}^{\infty}\subseteq(0,1)$ such that $\nu_i=\theta_i\mu_{m_i}+(1-\theta)\mu_{n_i}$ and $\int\varphi d\nu_i=a$ for any $i\in\{1,2,\cdots\}$. We can also choose proper $\kappa_1,\kappa_2,\in(0,1)$ such that $\rho_1=\kappa_1\mu_{m_1}+(1-\kappa_1)\mu_{n_2}, \rho_2=\kappa_2\mu_{m_1}+(1-\kappa_2)\mu_{n_3}$ and $\int\varphi d\rho_1=\int\varphi d\rho_2=a$. By proposition \ref{specification-full-support-dense}, there are $\mu^*,\mu^{**}$ with full support such that $\int\varphi d\mu^*<a<\int\varphi d\mu^{**}$. Choosing proper $\iota\in(0,1)$ such that $\mu=\iota\mu^*+(1-\iota)\mu^{**}$ and $\int\varphi d\mu=a$. Take $\omega_i=\frac{i-1}{i}\nu_1+\frac{1}{i}\nu_i,\ i\in\{1,2,\cdots\}$. Let

%So, there is $\{m_i\}_{i=1}^{\infty}$, $\{n_i\}_{i=1}^{\infty}$ such that for any
%
%
%such that $\int\varphi d\mu_1<\int\varphi d\mu_2<a<\int\varphi d\mu_3<\int\varphi d\mu_4<\int\varphi d\mu_5$ and $\int\varphi d\mu^*<a<\int\varphi d\mu^{**}$. So we can choose proper $\theta,\kappa,\tau,\rho,\iota\in(0,1)$ such that $\nu_1=\theta\mu_1+(1-\theta)\mu_3$, $\nu_2=\tau\mu_1+(1-\tau)\mu_4$, $\nu_3=\rho\mu_1+(1-\rho)\mu_5$, $\nu_4=\kappa\mu_2+(1-\kappa)\mu_3$, $\mu=\iota\mu^*+(1-\iota)\mu^{**}$ and $\int\varphi d\nu_1=\int\varphi d\nu_2=\int\varphi d\nu_3=\int\varphi d\nu_4=\int\varphi d\mu=a$. So $\cap_{i=1}^{3}S_{\nu_i}=S_{\mu_1}$, $S_{\nu_2}\cap S_{\nu_4}=\emptyset$ and $S_\mu=X$. Let
%
%
%Note that the measure satisfying Proposition \ref{distal-dense} and measures with full support are both dense in $\mathcal M_f(x)$. Then we can choose $\mu^*,\mu^{**}$ with full support and $\mu_1,\mu_2,\cdots$ satisfying Proposition \ref{distal-dense} and $\overline{\bigcup_{i=1}^{\infty} S_{\mu_i}}=X$.
\begin{align*}
%K_1\ &:=\ \ \nu_1\\
%K_2\ &:=\ \ \mathrm{cov}\{\nu_1,\nu_2\}\\
%K_3\ &:=\ \ \mathrm{cov}\{\nu_2,\nu_4\}\\
%K_4\ &:=\ \ \mathrm{cov}\{\nu_1,\nu_2\}\cup\mathrm{cov}\{\nu_1,\nu_3\}\\
%K_5\ &:=\ \ \mathrm{cov}\{\nu_1,\nu_2\}\cup\mathrm{cov}\{\nu_2,\nu_4\}
K_1\ &:=\ \ \mathrm{cov}\{\nu_1,\mu\},\\
K_2\ &:=\ \ \mathrm{cov}\{\nu_1,\mu\}\cup\mathrm{cov}\{\nu_1,\nu_2\},\\
K_3\ &:=\ \ \bigcup_{i=1}^\infty\mathrm{cov}\{\omega_i,\omega_{i+1}\},\\K_4\ &:=\ \ \bigcup_{i=1}^\infty\mathrm{cov}\{\omega_i,\omega_{i+1}\}\cup\mathrm{cov}\{\omega_1,\nu_2\},\\%\\
%\end{align*}
%\begin{align*}
K_5\ &:=\ \ \{ \nu_1\},\\
K_6\ &:=\ \ \mathrm{cov}\{\nu_1,\rho_1\},\\
K_7\ &:=\ \ \mathrm{cov}\{\nu_1,\nu_2\},\\
K_8\ &:=\ \ \mathrm{cov}\{\nu_1,\rho_1\}\cup\mathrm{cov}\{\nu_1,\rho_2\},\\
K_9\ &:=\ \ \mathrm{cov}\{\nu_1,\nu_2\}\cup\mathrm{cov}\{\nu_2,\nu_3\}.
\end{align*}
One can observe that $G_{K_i}\subseteq R_{\varphi}(a)$, for any $i\in\{1,2,3,4,5,6,7,8,9\}$. Based on the discussion in the proof of item (a), we complete the proof.
\qed

\bigskip

{\bf Proof of Item (d). } 
  If $Int(L_\varphi)\neq\emptyset$, then one can get this from item (c)  by taking one $a\in Int(L_\varphi)$ since $R_\varphi(a)\subseteq R_\varphi$. On the other hand, $Int(L_\varphi)=\emptyset$, then $R_\varphi=X$ and one can get this from item(a).\qed  %We will prove Theorems \ref{maintheorem-irregular} and \ref{maintheorem-levelsets} in Section \ref{section-mainproof}.

  \begin{Rem} If $(X,f)$ is not uniquely ergodic,  there are two different invariant measures $\mu,\nu$ so that by weak$^*$ topology there exists a continuous function $\phi$ such that $\int\phi d\mu\neq \int \phi d\nu.$ Thus $Int(L_\phi)\neq\emptyset$. Note that   $I_{\phi}(f)\neq\emptyset$ is equivalent to  $Int(L_\phi)\neq \emptyset$  if the system has specification property, see \cite{Thompson2008}. Thus item (a)  can also be deduced from item (b).

\end{Rem}

\subsection{Proof of Theorem \ref{Th-E-statisticallanguage}}
The proof is based on \cite[Theorem H]{HTW}. From the proof of \cite[Theorem H]{HTW}, we know that $$x\in BR\Leftrightarrow x\in\omega_{B^*}(x)\ \text{and}\ x\in QW\Leftrightarrow x\in\omega_{\overline{d}}(x).$$
The construction of $x$ in the proof of Theorem \ref{Th-best-refinedversion}   always satisfies that $x\in Trans\cap BR$, which implies $\omega_{B^*}(x)=\omega_f(x)=X$ by \cite[Lemma 4.{}6]{HTW}. Since the dynamical systems with specification are not minimal but minimal points are dense, so for any $x\in Trans,\ \omega_{B_*}(x)=\emptyset$. Thus one can check that the  uncountable DC1-scrambled sets constructed by $K_1, K_2, K_5, K_6, K_7$ in the proof of  Theorem \ref{Th-best-refinedversion}  satisfy the five cases, which ends the proof.
\qed

%\begin{Rem}

%\end{Rem}

\section{Applications}
\subsection{Examples with Specification}
It is known from \cite{BJ} that any topologically mixing interval map satisfies specification.  For example, \cite{JMV} showed that there exists a set of parameter values $\Lambda\subseteq[0,4]$ of positive Lebesgue measure such that if $\lambda\in\Lambda$, then the logistic map $f_\lambda(x)=\lambda x(1-x)$ is topological mixing.

 Moreover, maps satisfying the specification property includes the mixing subshift of finite type, mixing sofic subshift, topological mixing uniformly hyperbolic systems and the time-1 map of the geodesic flow of compact connected negative curvature manifolds, for example, see \cite{SigSpe,Thompson2008}. 

So, all the results of Theorems   \ref{maintheorem-irregular} - \ref{maintheorem-DC1inSaturated} are all suitable for such systems.

\subsection{Examples Without Specification}
Now, we use our theorem on a type of subshift which may not have specification property. Before the statement, we need some preparations.

For any finite alphabet $A$, the $full\ symbolic\ space$ is the set $A^{\mathbb{Z}}=\{\cdots x_{-1}x_{0}x_{1}\cdots : x_{i}\in A\}$, which is viewed as a compact topological space with the discrete product topology. The set $A^{\mathbb{N_{+}}}=\{x_{1}x_{2}\cdots : x_{i}\in A\}$ is called $one$-$side\ full\ symbolic\ space$. The $shift\ action$ on $one$-$side\ full\ symbolic\ space$ is defined by
$$\sigma:\ A^{\mathbb{N_{+}}}\rightarrow A^{\mathbb{N_{+}}},\ \ \ x_{1}x_{2}\cdots\mapsto x_{2}x_{3}\cdots.$$
$(A^{\mathbb{N_{+}}},\sigma)$ forms a dynamical system under the discrete product topology which we called a shift. A closed subset $X\subseteq A^{\mathbb{N_{+}}}$ is called $subshift$ if it is invariant under the shift action $\sigma$. $\mathbf{w}\in A^{n}\triangleq \{x_1x_2\cdots x_n:\ x_{i}\in A\}$ is a $word$ of $subshift$ $X$ if there is an $x\in X$ and $k\in \mathbb{N}$ such that $\mathbf{w}=x_kx_{k+1}\cdots x_{k+n-1}$. Here we call $n$ the length of $\mathbf{w}$ denoted by $|\mathbf{w}|$. The $language$ of a subshift $X$, denoted by $\mathcal L(X)$, is the set of all words of $X$. Denote $\mathcal L_{n}(X)\triangleq \mathcal L(X)\bigcap A^{n}$ all the words of $X$ with length $n$.

Now we introduce the typical subshift of one-side full shift space $\beta$-shift. Basic references are \cite{R,Sm,PS}. 
It is worth mentioning that from    \cite{BJ} the set of parameters of $\beta$ for which   specification holds, is dense in $ (1,+\infty)$ but has Lebesgue zero measure.

Let $\beta > 1$ be a real number. We denote by $[x]$ and $\{x\}$ the integer and fractional part of the real number $x$.
Considering the $\beta$-$transformation$ $f_{\beta}:[0,1)\rightarrow [0,1)$ given by $$f_{\beta}(x)=\beta x\ (\mathrm{mod}\ 1)$$
For $\beta \notin \mathbb{N}$, let $b=[\beta]$ and for $\beta \in \mathbb{N}$, let $b=\beta-1$. Then we split the interval $[0,1)$ into $b+1$ partition as below

$$J_0=\left[0,\frac{1}{\beta}\right),\ J_1=\left[\frac{1}{\beta},\frac{2}{\beta}\right), \cdots,\ J_1=\left[\frac{b}{\beta},1\right).$$
 For $x\in[0,1)$, let $i(x,\beta)=(i_n(x,\beta))_1^{\infty}$ be the sequence given by $i_n(x,\beta)=j$ when $f^{n-1}x\in J_j$. We call $i(x,\beta)$ the greedy $\beta$-$expansion$ of $x$ and we have
$$ x=\sum_{n=1}^{\infty}i_n(x,\beta)\beta^{-n}.$$
We call $(\Sigma_{\beta},\sigma)$ $\beta$-shift, where $\sigma$ is the shift map, $\Sigma_{\beta}$ is the closure of $\{i(x,\beta)\}_{x\in [0,1)}$ in $\prod_{i=1}^{\infty}\{0,1,\cdots,b\}$.

From the discussion above, we can also define the greedy $\beta$-$expansion$ of 1, denoted by $i(1,\beta)$. Parry showed that the set of sequence with belong to $\Sigma_{\beta}$ can be characterised as
$$\omega \in \Sigma_{\beta} \Leftrightarrow f^k(\omega) \leq i(1,\beta)\ \mathrm{for\ all}\ k\geq 1,$$
where $\leq$ is taken in the lexicographic ordering \cite{P}. By the definition of $\Sigma_{\beta}$ above, $\Sigma_{\beta_1}\subsetneq\Sigma_{\beta_2}$ for $\beta_1<\beta_2$(\cite{P}).   
 Now we introduce some lemmas about $\beta$-shift, which indicate that $\beta$-shift has a certain degree of transitive property.

%%\begin{definition}
%%The expansion of a number $x\in[0,1]$ in base $\beta$ is a sequence of integers out of $\{0,1,\cdots,[\beta]\}$,
%%$$\{i_n\}_1^{\infty}=\{i_n(x,\beta)\}_1^{\infty}$$
%%defined by one of the following equivalent properties
%%\item[(1)] $i_1=[\beta x]$, $i_2=[\beta\{\beta x\}]$, $i_3=[\beta\{\beta \{\beta x\}\}]$, $\cdots$;
%%\item[(2)] if $T_{\beta}:[0,1]\rightarrow [0,1)$ is the transformation $T_{\beta}(x)=\beta x(mod\ 1)$ then
%%$$ i_n=[\beta T_{\beta}^{n-1}(x)],\ n\in\mathbb{Z}.$$
%%\end{definition}

%%For sequences $\{i_n(x_1,\beta)\}_1^{\infty}$ and $\{i_n(x_2,\beta)\}_1^{\infty}$ the lexicographical order is defined by %%$\{i_n(x_1,\beta)\}_1^{\infty}<\{i_n(x_2,\beta)\}_1^{\infty}$ if and only if for the least index $s$ with $i_s(x_1,\beta)\neq i_s(x_2,\beta)$, $i_s(x_1,\beta) < i_s(x_2,\beta)$.

\begin{Lem}\label{lemma-g-product}

For any $\mathbf{w} \in \mathcal L_{n}(\Sigma_{\beta})$, if there is a $j\in[1,n]$ such that $\mathbf{w}_j\ne 0$, then for any $\eta \in \Sigma_{\beta}$, $\mathbf{w}_1\cdots(\mathbf{w}_j-1)\cdots\mathbf{w}_n\eta \in \Sigma_{\beta}$.

\end{Lem}

The proof is a easy part of \cite[Proposition 5.1]{PS}.

\begin{Lem}\label{lemma-connection}
For any $\omega \in \Sigma_{\beta}$ and any open set $U\subseteq \Sigma_{\beta}$, we can find an $\eta \in$ U and a $k\in \mathbb{N}$ such that $\sigma^k\eta=\omega$.
\end{Lem}

\begin{proof}
$U$ is open, so we can find a point $\xi=\xi_1\xi_2\cdots \in U$ such that $\xi < i(1,\beta)$. So we can find a $k\in \mathbb{N}$ large enough, such that $\xi^{'}\triangleq\xi_1\xi_2\cdots(\xi_k+1)\xi_{k+1}\xi_{k+2}\cdots<i(1,\beta)$ and $\xi^{'}\in U$. Then by Lemma \ref{lemma-g-product}, we conclude that $\eta \triangleq\xi_1\xi_2\cdots\xi_k\omega\in U$ and $\sigma^k\eta$ = $\omega$.
\end{proof}

\begin{Lem}\label{lemma-thompson}

For $\beta$-shift, there exists an increasing sequence $\{\Sigma_{\beta}^{n}\}$ of compact $\sigma$-invariant subsets of $\Sigma_{\beta}$ with the following properties:
\item[(a)] Each $\{\Sigma_{\beta}^{n}\}$ is a sofic shift and has specification property
\item[(b)] For any $\mu\in \mathcal M_{f}(\Sigma_{\beta})$, and any neighborhood $U$ of $\mu$ in $\mathcal M_{f}(\Sigma_{\beta})$, there exist $n\geq 1$ and $\mu^{'}\in \mathcal M_{f}^{e}(\Sigma_{\beta}^n)\bigcap U$.

\end{Lem}

Lemma \ref{lemma-thompson} is a main application in \cite{CTY}. Reader can refer to \cite{CTY} for the details of the proof. The lemma above shows us that to figure out the irregular set for the whole space($\Sigma_{\beta}$), it is sufficient to study the irregular set for certain asymptotic `horseshoe-like'($\Sigma_{\beta}^n$) of the whole space.

%\end{remark}

%\begin{definition}
%We use $X^{\beta}$ to represent the closure of the set of all $\beta$-expansions of $x\in[0,1]$.
%\end{definition}
%
%\begin{definition}
%The language $\mathcal{L}$ of the $\beta$-shift is the set of all words $\omega_1^k = (\omega_1,\cdots,\omega_k)$ with $\omega \in X^{\beta}$ and $k \in \mathbb{N}$.The set of all words of length k is denoted by $\mathcal{L}_k$
%\end{definition}
%
%\begin{Prop}\label{prop3}
%$\forall \omega_1^k \in \mathcal{L}_k$, if there is a $j\in[1,k]$ such that $\omega_j\ne 0$ ,then $\forall \eta \in X^\beta$ ,$(\omega_1,\cdots,\omega_j-1,\cdots,\omega_k,\eta) \in X^\beta$.
%\end{Prop}
%
%Readers can refer to [\cite{PS}, proposition 5.1] for details of the proof.
%
%
%\begin{lemma}\label{lemma-s}
%For any $\omega \in X^{\beta}$ and open set $U\subseteq X^{\beta}$, we can find a $\eta \in$ U such that there is a $k\in \mathbb{Z}_+$ and $\sigma^k\eta=\omega$.
%\end{lemma}
%
%
%\begin{proof}
%since $U$ is open, so we can find a point $\xi=(\xi_1,\xi_2,\cdots) \in$ U which is not the $\beta$-expansion of 1 which we label with $(c_1,c_2,\cdots)$. So we can find a large enough k, such that$(\xi_1,\xi_2,\cdots,\xi_k+1,\xi_{k+1},\xi_{k+2},\cdots)<(c_1,c_2,\cdots)$ which respect to the lexicographical order and it is still in U.Then we use the proposition \ref{prop3}, we conclude that $\eta = (\xi_1,\xi_2,\cdots,\xi_k,\omega)\in$ U and $\sigma^k\eta$ = $\omega$
%\end{proof}

\begin{Thm}\label{Th-beta1}
For any $\beta>1$ and $(\Sigma_{\beta},\sigma)$, suppose $\varphi$ is a continuous function on $\Sigma_{\beta}$, then
\begin{description}
\item[(a)] $\{QW_1, QW_2, QW_3, QW_4, QW_5, BR_1, BR_2, BR_3, BR_4, BR_5\}$ has  uncountable DC1-scrambled gap with respect to $ Rec(\sigma)$;

\item[(b)] If $I_{\varphi}(\sigma)\neq \emptyset$, then $\{QW_1, QW_2, QW_3, QW_4, QW_5, BR_1, BR_2, BR_3, BR_4, BR_5\}$ has  uncountable DC1-scrambled gap with respect to $ Rec(\sigma)\cap I_{\varphi}(\sigma)$;

\item[(c)] If $Int(L_\varphi)\neq\emptyset$, then for any $a\in Int(L_\varphi)$, $\{QW_1, QW_2, QW_3, QW_4, QW_5, BR_1, 
 BR_2, BR_3, BR_4,\\ BR_5\}$ has  uncountable DC1-scrambled gap with respect to $ Rec(\sigma)\cap R_{\varphi}(a)$.

 \item[(d)]    $\{QW_1, QW_2, QW_3, QW_4, QW_5, BR_1, 
 BR_2, BR_3, BR_4, BR_5\}$ has  uncountable DC1-scrambled gap with respect to $ Rec(\sigma)\cap R_{\varphi}$.

\end{description}
\end{Thm}

\begin{proof}
\textbf{(a)}: Refer to \cite{Sm}, we have $\{\beta\in (1,+\infty)\ |\ (\Sigma_{\beta},\sigma)\text{ has  specification  property} \}$ is dense in $(1,+\infty)$. Then for any $\beta>1$, we can find an $\alpha<\beta$ such that $(\Sigma_{\alpha},\sigma)$ has specification property. By Theorem \ref{Th-best-refinedversion}, for $(\Sigma_{\alpha},\sigma)$, we have $\{QW_1, QW_2, QW_3, QW_4, QW_5, BR_1,\\
 BR_2, BR_3, BR_4, BR_5\}$ has  uncountable DC1-scrambled gap with respect to $Trans_{\sigma|_{\Sigma_\alpha}}$. Note that $\Sigma_{\alpha}\subsetneq\Sigma_{\beta}$, so the transitive points of $\Sigma_{\alpha}$ must be the recurrent points of $\Sigma_{\beta}$. Moreover, it is easy to see that for any $S_a\in \mathrm{GS}\{QW_1, QW_2, QW_3, QW_4, QW_5, BR_1, BR_2,\\ BR_3, BR_4, BR_5\}$ for system $(\Sigma_{\alpha},\sigma)$ is a subset of some $S_b\in \mathrm{GS}\{QW_1, QW_2, QW_3, QW_4,\\ QW_5, BR_1, BR_2, BR_3,
  BR_4, BR_5\}$ for system $(\Sigma_{\beta},\sigma)$. Then item\textbf{(a)} has been proved.

\textbf{(b)}: If $I_{\varphi}(\sigma)\neq \emptyset$, there exist $\lambda_1,\lambda_2\in \mathcal M_\sigma(\Sigma_{\beta})$ such that $\int\varphi d\lambda_1\neq\int\varphi d\lambda_2$. By Lemma \ref{lemma-thompson}, we have $(\Sigma_{\beta}^n,\sigma)$ which has specification property and $\mu_1,\mu_2\in \mathcal M_\sigma(\Sigma_{\beta}^n)$ such that $\int\varphi d\mu_1\neq\int\varphi d\mu_2$. By the proof of Theorem  \ref{Th-best-refinedversion}, for $(\Sigma_{\beta}^n,\sigma)$, we have $\{QW_1, QW_2, QW_3, QW_4,\\
 QW_5, BR_1, BR_2, BR_3, BR_4, BR_5\}$ has  uncountable DC1-scrambled gap with respect to $Trans_{\sigma|_{\Sigma_{\beta}^n}}\cap I_{\varphi}(\sigma)$. Similarly, item\textbf{(b)} has been proved.

\textbf{(c)}: If $Int(L_\varphi)\neq\emptyset$, then for any $a\in Int(L_\varphi)$, there exist $\lambda_1,\lambda_2$ such that $\int\varphi d\lambda_1<a<\int\varphi d\mu_2$. By Lemma \ref{lemma-thompson}, we have $(\Sigma_{\beta}^n,\sigma)$ which has specification property and $\mu_1,\mu_2\in \mathcal M_f(\Sigma_{\beta}^n)$ such that $\int\varphi d\mu_1<a<\int\varphi d\mu_2$. By the proof of Theorem  \ref{Th-best-refinedversion}, for $(\Sigma_{\beta}^n,\sigma)$, we have $\{QW_1, QW_2, QW_3, QW_4, QW_5, BR_1, BR_2, \\ BR_3, BR_4, BR_5\}$ has  uncountable DC1-scrambled gap with respect to $Trans_{\sigma|_{\Sigma_{\beta}^n}}\cap R_{\varphi}(a)$. Similarly, item\textbf{(c)} has been proved.

\textbf{(d)}: If $Int(L_\varphi)\neq\emptyset$, item\textbf{(d)} is from item\textbf{(c)}. Otherwise, $R_\varphi=X$ so that item\textbf{(d)} is from item\textbf{(a)}.
\end{proof}

\section{Comments and Questions}

\subsection{Weakly almost periodic points} 
The reason why we can't analyse whether there is an   uncountable DC1-scrambled set in $W$ by our method is that we didn't find a measure $\mu$ with full support and $G_\mu$ has distal pair. For a point $x\in W\cap Trans$, we can observe that $x$ must be a element of the generic point of a measure with full support. But   Theorem \ref{maintheorem-DC1inSaturated} don't cover this situation. 

\begin{Thm}\label{Th-Weaklyperiodic}
Suppose that $(X,f)$ has specification property. If 
%there are two invariant measures $\mu,\nu$ with full support (admitting $\mu=\nu$) and   %$x,y\in Trans$ and 
%$  G_\mu,   G_\nu,$ both have distal pair, 
 for any invariant measure  $\mu$ with full support,    %$x,y\in Trans$ and 
$  G_\mu $   has distal pair, then \\ 
(1) there is an   uncountable DC1-scrambled set $S\subseteq W\cap Trans.$
\\
(2) If $\varphi$ is a continuous function on $X$ and $I_{\varphi}(f)\neq\emptyset$. Then there is an   uncountable DC1-scrambled set $S\subseteq W\cap Trans\cap I_{\varphi}(f).$\\
(3) If $\varphi$ is a continuous function on $X$ and $Int(L_\varphi)\neq\emptyset$. Then for any $a\in L_\varphi$, there is an   uncountable DC1-scrambled set $S\subseteq W\cap Trans\cap R_{\varphi}(a)$. 
\\
(4) For  any continuous function $\varphi$   on $X$,  there is an   uncountable DC1-scrambled set $S\subseteq W\cap Trans\cap R_{\varphi}$.
\end{Thm}

\begin{Rem} The set of  points with Case (1) restricted on recurrent set coincides with the set of $W\setminus AP.$ For systems with specification, note that $W\cap Trans \subseteq W\setminus AP$ so that above result can be also stated for the set of  points with Case (1) restricted on recurrent set or $W\setminus AP.$

\end{Rem}

\begin{Rem} For a transitive system $(X,f)$ without periodic points with period $m,$ it is easy to check for any $x\in Trans$, $(x,f^mx)$ must be a distal pair. This implies that for any invariant measure $\mu$ (not necessarily with full support),  
$  G_\mu\cap Trans $   has distal pair.  So Theorem \ref{Th-Weaklyperiodic} are suitable for systems with specification but without periodic points with period $m$ for some $m.$ In particular, it apllies in mixing subshifts of finite type without periodic points with period $m$ for some $m.$  For example it can be a subshift of finite type defined by a graph
with two distinct cycles of length $m+1$ and $m+2$ starting from the same vertex. For such dynamical systems, Theorem \ref{maintheorem-DC1inSaturated} holds for any nonempty compact connected set $K,$ since $G_\mu$ has distal pair for any $\mu$ in $K.$
\end{Rem}

{\bf Proof.}  Let $\mu$ be an invariant measure with full support.

(1) Take $K= \{\mu\}$. Then one can use Proposition \ref{prop1} and Theorem \ref{maintheorem-DC1inSaturated} to give the proof. 

(2) By Proposition \ref{specification-full-support-dense}, one can choose an invariant measure $\mu'$ with full support such that $\int \varphi d\mu \neq \int \varphi d\mu'.$
 Take $K=\text{cov}\{\mu,\mu'\}$. Then one can use Proposition \ref{prop1} and Theorem \ref{maintheorem-DC1inSaturated} to give the proof.

 (3) If $\int \varphi d\mu=a$, take $\omega=\mu.$ Otherwise, by Proposition \ref{specification-full-support-dense}, one can choose an invariant measure $\mu'$ with full support such that $\int \varphi d\mu'<a< \int \varphi d\mu$ or $\int \varphi d\mu<a< \int \varphi d\mu'.$   Take suitable $\theta\in(0,1)$ such that $\omega=\theta\mu+(1-\theta)\mu'$ such that $\int \varphi d\omega=a. $ In this case take $K=\{\omega\}$. Then one can use Proposition \ref{prop1} and Theorem \ref{maintheorem-DC1inSaturated} to give the proof.

(4) If $Int(L_\varphi)\neq\emptyset$, item  {(4)} is from item  {(3)}. Otherwise, $R_\varphi=X$ so that item  {(4)} is from item  {(1)}.
\qed

\subsection{Minimal points}\label{section-minimal}

For minimal points, it is still unknown whether DC1 appear   but we remark that DC-2 appear.
\begin{Thm}\label{Th-Minimal}
Suppose that $(X,f)$ has specification property (or almost specification, or shadowing property with positive entropy). Then   there is an uncountable  DC-2 scrambled set $S\subseteq   AP(f).$ 
 \end{Thm}

{\bf Proof.} From \cite{Dwic} a dynamical system with positive entropy has DC-2 scrambled set so that if a minimal subsystem has positive entropy, then the proof is completed. In fact, from \cite{DongTian2017}, we know there exist minimal subsystems arbitrarily close to full entropy (and thus $  AP(f)$ carries full topological entropy). \qed

\medskip

 From    \cite{BJ} the set of parameters of $\beta$ for which   specification holds, is dense in $ (1,+\infty)$ but has Lebesgue zero measure. However,  every $\beta$ shift has almost specification by \cite{PS2} so that  Theorem \ref{Th-Minimal} applies in all $\beta$ shifts. 

Let    $C(M) $be the set of continuous maps on a compact manifold $ M$ and
$H(M)$ the set of homeomorphisms on $M$.    Recall that $C^0$ generic $f\in H(M)$ (or $f\in C(M)$) has the shadowing property and infinite topological entropy (see \cite{KMO2014} and \cite{ Koscielniak2,  Koscielniak},   respectively). Thus   Theorem \ref{Th-Minimal} applies in $C^0$ generic dynamical systems.

\bigskip

{\bf Acknowledgements.  }   Tian is the corresponding author and is supported by National Natural Science Foundation of China (grant no.   11671093).

\end{document}